\documentclass[11pt,a4paper]{article}
\usepackage{amsthm}
\newtheorem{theorem}{Theorem}
\newtheorem{lemma}{Lemma}
\newtheorem{corollary}{Corollary}
\newtheorem{assumption}{Assumption}
\newtheorem{definition}{Definition}
\newtheorem{remark}{Remark}
\newtheorem{example}{Example}
\usepackage{style}
\usepackage{thm-restate}
\usepackage[authoryear,round]{natbib}
\graphicspath{{./}}
\counterwithout{equation}{section}
\counterwithout{figure}{section}
\crefname{enumi}{}{}
\crefname{equation}{}{}
\crefname{assumption}{Assumption}{Assumptions}
\Crefname{assumption}{Assumption}{Assumptions}
\def \R {\mathbb {R}}

\def \E {\mathbb{E}}
\def \P {\mathbb{P}}
\def \V {\mathbb{V}}

\def \Ent {\mathrm{Ent}}
\newcommand{\IID}{{\rm i.i.d.}}
\newcommand{\iid}{\overset{\IID}{\sim}}

\renewcommand{\tilde}{\widetilde}

\renewcommand{\eps}{\varepsilon}

\allowdisplaybreaks

\title{Concentration inequalities for leave-one-out cross-validation}
\author{Benny Avelin\thanks{Department of Mathematics, Uppsala University, S-751 06 Uppsala, Sweden. \texttt{benny.avelin@math.uu.se}.}
\and Lauri Viitasaari\thanks{Department of Information and Service Management, Aalto University School of Business, 00076 Aalto, Finland. \texttt{lauri.viitasaari@aalto.fi}.}}
\date{}
\begin{document}
\maketitle
\begin{abstract}
In this article we prove that estimator stability, together with a gradient-growth condition on the expected fitted loss, is enough to provide concentration bounds for leave-one-out cross-validation in a general framework. In particular, we provide concentration bounds beyond Lipschitz continuity assumptions on the loss or on the estimator. We obtain our results by relying on random variables with distribution satisfying the logarithmic Sobolev inequality, providing us a relatively rich class of distributions. We illustrate our method by considering several interesting examples, including linear regression, kernel density estimation, and stabilized/truncated estimators such as stabilized kernel regression.
\end{abstract}

\noindent\textbf{Keywords:} Leave-one-out cross-validation, concentration inequalities, logarithmic Sobolev inequality, sub-Gaussian random variables, estimator stability.

\tableofcontents

\section{Introduction}
It is customary in many statistical and machine learning methods to use train-validation, where the data are split into a training set and a validation set. The training set is used for estimating the model, while the validation set is used to measure the performance of the fitted model.
The drawback of a fixed train-validation split is that a substantial part of the data is reserved for only one of these two tasks: using too little data for training can reduce accuracy, while using too little data for validation can give an unreliable assessment of performance. One common technique to overcome this problem is to perform cross-validation, originally introduced by Stone in~\cite{Stone}. In the present paper, we consider the classical \emph{leave-one-out} ({LOO}) cross-validation. In {LOO} cross-validation, one leaves one fixed observation out and uses the remaining observations as the training set. The omitted observation is then used as the validation point. Repeating this procedure through the whole data set of size $n$ gives a performance estimate obtained by averaging over all $n$ validation points. Thus each step uses almost all the data for training, while the final performance estimate still evaluates the fitted procedures on all observations.

{
Cross-validation has many uses, but the most common ones are as a model selection procedure (see~\cite{Stone}) or as a method to estimate generalization error.
In this paper we are interested in estimating the generalization error. For a discussion about model selection, see \cref{sec:conclusion}.
}

Concentration and algorithmic stability measure how well one estimates the true error, with respect to a given loss function, arising from fitting the model in the first place, and hence are crucially important topics for practical applications. Indeed, under- or over-fitting can lead to severe inaccuracies in predictions made by the fitted model, thereby reducing prediction performance. The topic has therefore been extensively studied in the literature from different perspectives. Several theoretical studies on the goodness of cross-validation include, among others,~\cite{Bauer-Reiz,Becker,cornec,Golub-Heath-Wahba,Gu,Gyorfi-Kohler-Krzyzak-Walk,Kindermann-Neubauer,Kindermann-Pereverzyev-Pilipenko,Li,Lukas}
in various settings. The $L^2$-error is studied, under various settings, in~\cite{Bartel-Hielscher,Bates-et-al,Gyorfi-Kohler-Krzyzak-Walk,Kale-Kumar-Vassilvitskii,Markatou-et-al}. For the results related to algorithmic stability, we refer to
~\cite{BE, CG, DGL, DW79, DW79b, KN, RMP, SSSSS, AMS,Kale-Kumar-Vassilvitskii,Kearns-Ron,Kumar-Lokshtanov-Vassilvitskii} and references therein. For a very good history of the subject of algorithmic stability in the context of cross-validation, see~\cite{AMS}. Finally, while we do not discuss the issue of cross-validation having relatively heavy computational burden in the present paper, the reader may consult~\cite{Bartel-et-al,Deshpande-Girard,Lukas-Hoog-Anderssen,Sidje-Williams-Burrage,Weinert} for efficient computations.

There is also a substantial line of work on concentration for functions of independent variables beyond bounded-difference assumptions. In particular, the sub-Gaussian diameter of an unbounded metric probability space was introduced in~\cite{Kontorovich2014} and applied there to algorithmic stability with unbounded losses. In that work, the bounded diameter entering McDiarmid-type arguments is replaced by a distribution-dependent refinement of the metric diameter, leading to concentration and generalization bounds in metric spaces whose ordinary diameter may be infinite. Our assumptions are of a different form: the data live in a Euclidean product space with a log-Sobolev distributional assumption, and the sensitivity of the fitted loss is measured by the data-coordinate gradient bound in \cref{definition:standing}. Thus we do not use the sub-Gaussian diameter as the stability quantity; instead, the terms $\delta_1$ and $\delta_2$ quantify Lipschitz and linear-growth components of the gradient of $D\mapsto L(T_D(x),y)$, while $\delta_3$ controls the delete-one bias specific to LOO cross-validation. Sub-Gaussian and sub-exponential analogues of McDiarmid's inequality for general functions of independent random variables, formulated through centered conditional versions and Orlicz norms and proved by the entropy method, were developed in~\cite{MaurerPontil2021}. This latter work is close in spirit to the concentration input used below, while both references show that the general move beyond bounded-difference concentration is not new by itself. In the proof section, \cref{thm:subexp-concentration,lemma:quadratic,cor:quadconc} serve as self-contained concentration tools tailored to the log-Sobolev and gradient-growth assumptions used in the present paper, rather than as new general replacements for the conditional-Orlicz theory of~\cite{MaurerPontil2021}. Our contribution is in a different direction: we specialize such concentration ideas to leave-one-out cross-validation, decompose the LOO error into the three terms appearing in \cref{rem:errors}, and express the resulting rates through explicit stability quantities for the fitted loss.

Despite the problem being well-studied in the literature, concentration bounds for {LOO} cross-validation are less developed in a general unbounded framework. Most existing results assume simplifying properties such as boundedness of the data or loss, e.g.~\cite{cornec}, or Lipschitz continuity in the Hamming distance, e.g.~\cite{Bartel-Hielscher}. In this article we study concentration inequalities for {LOO} cross-validation in a general unbounded framework using gradient bounds on the fitted loss. For the data, we assume independent observations from a distribution satisfying the logarithmic Sobolev inequality, prototypically the Gaussian. This assumption gives concentration for Lipschitz functions through the Herbst argument. The role of algorithmic stability is then to identify the relevant Lipschitz quantity: in the {LOO} setting, the gradient of the fitted loss gives a quantitative measure of how sensitive the fitted predictor is to perturbations of the data.
We finally mention an interesting and related paper,~\cite{AMS}, where the authors assume a Gaussian type moment bound on the {LOO} estimate, allowing for unbounded data. In comparison, for Gaussian data this allows one to use Lipschitz loss functions, while our framework allows us to go beyond and cover quadratic losses as well, with a small price in the exponential rate.

The logarithmic Sobolev inequality has its roots in the theory of hypercontractivity of the Ornstein-Uhlenbeck semigroup, and was essentially first discovered by L. Gross in~\cite{Gross}. The concentration of measure on the other hand was a term coined by Milman (for references see~\cite{Milman}) in his study of asymptotic theory of Banach spaces, essentially saying that most of the surface area of the sphere is concentrated on the equator. However, it was Herbst who made the connection between the log-Sobolev inequality and concentration of measure in an unpublished letter to Gross, for more information see~\cite{Ledoux}. The logarithmic Sobolev inequality is useful, because it tensorizes and is thus very suitable for the i.i.d.~framework. The assumption is rather mild in the sense that many distributions satisfy it. On the other hand however, it imposes Gaussian-type concentration.

\paragraph{Lean formalization.} A formalization of the main results in the Lean~4 proof assistant is available in the \href{https://github.com/BennyAvelin/leave-one-out-concentration}{\texttt{leave-one-out-concentration}} repository on GitHub. It covers \cref{thm:simplified,thm:main,thm:data-dependent,cor:loo-moments} and the ordinary least-squares concentration bound in \cref{ssec:linear}. The formal proofs build on mathlib and contain no unproved placeholders or custom axioms. The formal statements and their correspondence with the mathematical definitions and hypotheses were independently audited. The development was carried out by AI agents under the supervision of the authors, with separate agents reviewing the statements and proofs.

\paragraph{Organization of the paper.}
The rest of the article is organized as follows. In \cref{sec:main:ass} we introduce the LOO notation, the logarithmic Sobolev and Poincaré tools, and the gradient-based stability assumptions used throughout the paper. In \cref{sec:main} we state the main concentration bounds, beginning with the Lipschitz warm-up, followed by the general gradient-growth result and the conditioned extension. In \cref{sec:examples} we verify the assumptions in several examples, including linear regression, kernel density estimation, and stabilized/truncated estimators such as stabilized kernel regression. Numerical experiments are reported in \cref{sec:numerical}. The proofs are given in \cref{sec:proofs}, and \cref{sec:conclusion} discusses limitations, model selection, and possible extensions to other cross-validation schemes.

\section{Setup and preliminaries for leave-one-out cross-validation}\label{sec:main:ass}
In order to define the leave-one-out ({LOO}) cross-validation we need first a definition:
Throughout the paper, we assume that $n\geq2$.
\begin{definition}
	Let $Z_1,\ldots,Z_n$ be an i.i.d.~sequence of random variables, $Z_i = (X_i,Y_i) \in \R^k \times \R^m$. We will call $D = (Z_1,\ldots,Z_n)$ a \emph{dataset} of length $n$. Furthermore, for any $i = 1,\ldots,n$ we define the \emph{deleted dataset} $D_{(-i)} = (Z_1,\ldots,Z_{i-1},Z_{i+1},\ldots,Z_n)$.
\end{definition}

Leave-one-out cross-validation is often used to evaluate the prediction performance of a statistic.
\begin{definition}
	For a statistic $T_n: \R^{(k+m)n} \to C(\R^k; \R^m)$ and a dataset $D$ of length $n$, we denote $T_{D} = T_n[D] \in C(\R^k; \R^m)$. We also denote $T_{D_{(-i)}} = T_{n-1}[D_{(-i)}]$.	Furthermore, we call any function {$L: \R^m \times \R^m \to \R$} a \emph{loss function}.
	The \emph{risk} is the expected loss w.r.t.~a random variable $Z = (X,Y)$ which is independent of $D$ and has the same law as $(X_i,Y_i) \in D$.
\end{definition}

\begin{definition}\label{definition:LOO}
	Given a statistic $T_n: \R^{(k+m)n} \to C(\R^k; \R^m)$, a dataset $D$ of length $n$, and a loss function { $L: \R^m \times \R^m \to \R$} we define the \emph{leave-one-out} ({LOO}) cross-validation of risk as
	\begin{align*}
		\hat L_{{LOO}} := \frac{1}{n} \sum_{i=1}^n L(T_{D_{(-i)}}(X_i),Y_i) \sim L_{{LOO}} := \E[L(T_D(X),Y) \mid D],
	\end{align*}
	where $Z = (X,Y)$ is independent of $D$ and has the same law as each observation, and $\sim$ means `estimate of'.
\end{definition}

Although the introduction of $L$ and $T_D$ is necessary, the notation is cumbersome. Throughout the paper, for $z=(x,y)$, we use the shorthand notation
\begin{align*}
	f_i(z) = L(T_{D_{(-i)}}(x),y), \qquad f_D(z) = L(T_D(x),y).
\end{align*}
Thus $f_i$ is the loss evaluated using the estimator trained without the $i$th observation, while $f_D$ is the loss evaluated using the estimator trained on the full data set. In the main estimates and proofs we keep this shorthand notation, but we expand it at the first occurrence whenever the dependence on the deleted data set or the full data set is important.

This is an \emph{a posteriori} measurement of error (conditioned on $D$) and we consider the model ($T$) as fitted and fixed, we are just testing its prediction performance. Our main theorems below provide concentration inequalities for the error $L_{{LOO}}-\hat L_{{LOO}}$.

In what follows, we use $|\cdot|$ to denote the absolute value and $\|\cdot\|$ to denote the Euclidean norm given by $\|x\|^2 = \sum_{k=1}^n x_k^2$, where $n$ is the size of the vector. We identify a data set $D=(Z_1,\ldots,Z_n)$ with the corresponding vector in $\mathbb{R}^{(k+m)n}$. Consequently, $\nabla_D$ denotes the gradient with respect to all data coordinates $(Z_1,\ldots,Z_n)$, and $\|\nabla_D f_D(z)\|$ is the Euclidean norm of this full gradient. For locally Lipschitz maps, gradient bounds below are understood Lebesgue-almost everywhere on the full Euclidean domain.

\subsection{Logarithmic Sobolev inequalities, Poincaré control, and concentration}

Throughout, we consider our data to arise from a distribution that has a finite logarithmic Sobolev constant, given below:
The observations have law $\mu$ on $\R^{k+m}$ and satisfy $\E[\|Z\|^2]<\infty$.
\begin{definition}
	We define the logarithmic Sobolev constant of a probability measure $\mu$ as
	\begin{align*}
		\sigma^2(\mu) := \sup_{f \in W^{1,2}(\mu)} \frac{1}{2}\frac{\Ent_\mu(f^2)}{\E_\mu[\|\nabla f\|^2]},
	\end{align*}
	where the entropy of $g$ is given by
	\begin{align*}
		\Ent_\mu(g) = \E_\mu[g \log(g)] - \E_\mu[g]\log(\E_\mu[g]),
	\end{align*}
	and $W^{1,2}(\mu)$ is the Sobolev Hilbert space obtained by completing bounded smooth ambient functions with bounded first derivatives, modulo zero norm, in $(\E_\mu[f^2]+\E_\mu[\|\nabla f\|^2])^{1/2}$, retaining both functions and gradients when $\mu$ is singular. Functions with zero gradient energy are omitted from the supremum.
\end{definition}

We assume that $\sigma^2(\mu)<\infty$. If $\sigma^2(\mu)=0$, then $\mu$ is a point mass and the concentration bounds below have their limiting interpretation; in the proofs we take $\sigma^2(\mu)>0$. Thus
\begin{equation}\label{eq:lsi}
 \Ent_\mu(f^2)\leq 2\sigma^2(\mu)\E_\mu[\|\nabla f\|^2],\qquad f\in W^{1,2}(\mu).
\end{equation}
The inequality is assumed for bounded truncations of the functions used below as well.

Thus the definition compares an entropy term, which measures the spread of $f^2$, with the average squared local variation of $f$. In this sense the logarithmic Sobolev inequality is a functional way to turn gradient bounds into concentration estimates.

{The log-Sobolev constant essentially measures the tail behavior of the measure, but it does so in a rather strong way.}
The log-Sobolev {constant} is related to the {interplay between the geometry of the manifold of support and the log-concavity of the density at the tails.} Examples include measures uniformly bounded from below and above on the unit sphere and the Gaussian measure.
However, this is still a fairly rich class of densities, and covers densities that are strongly log-concave outside a large ball under suitable perturbation conditions. See also~\cite{Ledoux} for Bakry-Emery and Holley-Stroock perturbations. Worth mentioning is also the Aida perturbation argument, see~\cite{Aida}.

The first consequence we use is concentration. The fact that $\sigma^2(\mu) < \infty$ implies concentration is well known, see for instance~\cite{Ledoux}. More concretely, the Herbst argument applied to the logarithmic Sobolev inequality shows that every $1$-Lipschitz function is sub-Gaussian with variance proxy controlled by $\sigma^2(\mu)$. That is, if $F$ is smooth enough and $\|\nabla F\|\leq 1$, then
\begin{align*}
	\P\left ( F-\E[F] > t \right ) \leq \exp\left(-\frac{t^2}{2\sigma^2(\mu)}\right), \quad t>0,
\end{align*}
up to the convention used for the log-Sobolev constant. This is the elementary mechanism behind the use of logarithmic Sobolev inequalities below: if the fitted loss has a small gradient as a function of the data, then the fitted loss cannot fluctuate much under the data distribution.

\begin{lemma}\label{lemma:sgls}
	Let a probability measure $\mu$ satisfy $\sigma^2(\mu) < \infty$. Then for \emph{the sub-Gaussian constant} of a measure $\mu$, defined as
	\begin{align*}
		\sigma_{SG}^2(\mu) = \inf\left \{\sigma^2 : \int e^{t(f-\int f\,d\mu)} d\mu \leq e^{\sigma^2 t^2 /2}, \quad \forall t \in \R, \quad \forall f\in L^1(\mu):\operatorname{Lip}(f)\leq 1 \right \},
	\end{align*}
	we have $\sigma_{SG}^2(\mu) \leq \sigma^2(\mu)$.
\end{lemma}

The next point is tensorization. The log-Sobolev property tensorizes with the same constant for product measures, i.e.~$\sigma^2\left(\mu^{\otimes n}\right) = \sigma^2(\mu)$. In the present i.i.d.~setting, this means that if each observation has law $\mu$, then the joint law of $D=(Z_1,\ldots,Z_n)$ is the product measure $\mu^{\otimes n}$, and the same log-Sobolev inequality holds for functions of the whole data set. Thus, the constant does not deteriorate merely because we pass from one observation to an i.i.d.~sample of length $n$, which is why the assumption is well suited to the cross-validation setting.

Besides tail bounds, the same functional-inequality framework also gives variance control. In particular, the log-Sobolev inequality implies a Poincaré inequality (see~\cite{Ledoux}): for any smooth enough function $F$ of the data,
\begin{align*}
	\V[F] \leq \sigma^2(\mu)\E[\|\nabla_D F\|^2],
\end{align*}
where $\sigma^2(\mu)$ is the log-Sobolev constant. Since the log-Sobolev constant tensorizes, the same type of estimate applies to functions on the product space carrying the i.i.d.~dataset. Consequently, if
\begin{align*}
	\|\nabla_D F(D)\| \leq \delta_1(n) + \delta_2(n)\|D\|,
\end{align*}
and $\E[\|D\|^2]\approx n$, then the Poincaré bound gives the heuristic scale
\begin{align*}
	\V[F] \lesssim \delta_1^2(n) + n\delta_2^2(n),
\end{align*}
up to constants depending on the distributional parameters. Thus the term $\delta_2(n)\|D\|$ is meaningful for concentration only when $\delta_2(n)$ decays faster than the natural $\sqrt n$ growth of $\|D\|$.

In the applications below, the random function to which the concentration and variance estimates are applied is usually the fitted-loss map
\begin{align*}
	D \mapsto L(T_D(x),y),
\end{align*}
for a fixed test point $z=(x,y)$, or an average of such maps. This is why the stability assumption is stated as a gradient bound on $\|\nabla_D L(T_D(x),y)\|$: it is exactly the quantity inserted into the log-Sobolev, Poincaré, and concentration estimates.

\subsection{Algorithmic stability}

The main objects of study will be the following concept of stability.
\begin{definition}\label{definition:standing}
	If there exist functions $\delta_1(n,z),\delta_2(n,z),\delta_3(n)$ all non-negative and tending to zero as $n \to \infty$, such that for Lebesgue-almost every $D \in \mathbb{R}^{(k+m)n}$ and $\mu$-almost every $z = (x,y) \in \R^{k+m}$, {it holds}
	\begin{align}
		\label{eq:gradient-condition}
		\|\nabla_D L(T_{D}(x),y)\| \leq \delta_1(n,z) + \delta_2(n,z)\|D\|,
	\end{align}
	and the fitted losses are jointly measurable and $D\mapsto L(T_D(x),y)$ is locally Lipschitz for $\mu$-almost every $z$,
	and for a dataset $D$ of length $n$ whose distribution is $\mu^{\otimes n}$ with $\sigma^2(\mu) < \infty$, it holds
	{
	\begin{align}\label{eq:delta3}
		|\E[L(T_{D_{(-i)}}(X_i),Y_i)]-\E[L(T_{D}(X),Y)]| \leq \delta_3(n), \quad i = 1,\ldots,n.
	\end{align}}
	Here $Z=(X,Y)$ is independent of $D$ with law $\mu$, and the displayed expectations are finite. Convergence of $\delta_1(n,z)$ and $\delta_2(n,z)$ is understood for $\mu$-almost every fixed $z$. Then we say that the pair $(T,L)$ is $(\delta_1,\delta_2,\delta_3)$-stable.
\end{definition}

	{
	\begin{remark}\label{rem:variance}
		The quantities $\delta_1$ and $\delta_2$ should be read as the two rates in the gradient bound
		\begin{align*}
			\|\nabla_D L(T_D(x),y)\| \leq \delta_1(n,z)+\delta_2(n,z)\|D\|.
		\end{align*}
		The first term is the Lipschitz-type part, while the second permits controlled linear growth in the size of the data vector. The factor $\|D\|$ is not normalized in the definition; for typical i.i.d.~data its size is of order $\sqrt n$, so the required normalization is instead encoded in the decay rate of $\delta_2(n,z)$. The quantity $\delta_3$ is of a different nature: it measures the delete-one bias in \cref{eq:delta3}.
	\end{remark}

	Let us frame \cref{definition:standing} in a particular example. We first consider the simplest possible:
	\begin{example}\label{example:mean}
		Let $D$ be a dataset of length $n$. Let $T_D$ be the empirical mean. Then, $T_D$ is Lipschitz with respect to $D$. Let $L(x,y)=(x-y)^2$ be the quadratic loss. Then, the loss function evaluated on the estimator ($|T_D(x)-y|^2$) now satisfies \cref{definition:standing} with $\delta_1 \sim 1/\sqrt{n}$, $\delta_2 \sim 1/n$, and $\delta_3 \sim 1/n$. For a complete derivation, see \cref{sec:examples}.
	\end{example}
	On the other hand, a classical but a bit more complicated example is the following:
	\begin{example}\label{example:kernel}
		Consider kernel density estimation with a given kernel $K$ in $\mathbb{R}$. Assume that the kernel $K$ is differentiable and has a globally bounded derivative. We also assume $K\geq0$. Let our i.i.d.~dataset of length $n$ be denoted as $D$. Then, denote the kernel density estimator $\hat g_D$ with bandwidth $h$ of a density $g$. To estimate the integrated $L^2$ loss, one needs to estimate $L_{LOO} := \mathbb{E}[\hat g_D(Z) \mid D]$, where $Z$ is a new data point. This is often done using leave-one-out cross-validation, which involves computing the corresponding $\hat L_{LOO}$ as defined in \cref{definition:LOO}, where the estimator $T_D(z) = \hat g_D(z)$ and the loss is $L(T_D(z)) = |T_D(z)|$. In \cref{sec:examples}, we will see that $\delta_1(n) = C/(h^2 \sqrt{n})$, $\delta_2 \equiv 0$, and $\delta_3 \equiv 0$ when the same bandwidth is used in the full and deleted fits.
	\end{example}

		In \cref{sec:examples}, we will see that $\delta_1(n) \sim 1/\sqrt{n}$ in many examples (which corresponds to the best case, as seen in \cref{example:mean,example:kernel}). For the growth term, the Poincaré heuristic above shows that one needs at least $\delta_2 = o(\sqrt{n}^{-1})$ for the variance contribution to vanish as $n \to \infty$.

	To the authors, it seems that the introduction of $\delta_1$ and $\delta_2$ in the LOO setting is new, while the error controlled by $\delta_3$ appears in the literature. Historically, the error controlled by \cref{eq:delta3} has been measured in different norms, the strongest being uniform stability introduced in~\cite{BE}. A uniform bound on the change of the loss under deletion implies \cref{eq:delta3} after taking expectations. Weaker notions of stability have been studied; for instance, $L^1$ (hypothesis stability) and $L^2$ (mean square stability) in~\cite{DW79b,Kale-Kumar-Vassilvitskii}, respectively. Uniform stability controls the worst possible change of the loss under a modification of the training data. Hypothesis stability and mean-square stability replace this worst-case control by average $L^1$ or $L^2$ controls. Error stability, first introduced in~\cite{Kearns-Ron}, is weaker still and controls the corresponding expected error. Our formulation in \cref{eq:delta3} is of this error-stability type.

	Under uniform stability for bounded losses, one gets Gaussian concentration bounds~\cite{BE}. This remains true if one has sufficient control of all the moments~\cite{AMS}. Under our weaker assumption on \cref{eq:delta3}, together with the gradient control encoded by $\delta_1$ and $\delta_2$ and the expected-loss condition stated below, we obtain either Gaussian or exponential concentration. In this sense the gradient condition does not replace the existing stability hierarchy. Rather, $\delta_3$ controls the delete-one bias, while $\delta_1$ and $\delta_2$ supply the additional fluctuation control of the fitted loss as a function of all data coordinates.
}

\section{Main results}\label{sec:main}

We now state the concentration bounds. We begin with a simplified Lipschitz case, \cref{thm:simplified}, where the estimates have a purely Gaussian form and the three sources of error are easiest to see. We then introduce the rate functions needed for the general quadratic-gradient framework and state the main estimate, \cref{thm:main}. Finally, \cref{thm:data-dependent} treats the case where the Lipschitz control is available only after conditioning on a large set of well-behaved data configurations.
Let $f(z)=\E[f_i(z)]$ for every $i$, since each $D_{(-i)}$ has law $\mu^{\otimes(n-1)}$. The full and deleted fits use the stability parameters at sizes $n$ and $n-1$, respectively.

The following standing assumption collects the hypotheses used throughout the section.

\begin{assumption}\label{assumption:standing}
	We assume that $D \in \mathbb{R}^{(k+m)n}$ is a dataset of length $n$ whose distribution is $\mu^{\otimes n}$, with $\sigma^2(\mu) < \infty$ and $\E[\|Z\|^2]<\infty$. We furthermore assume that the pair $(T,L)$ is $(\delta_1,\delta_2,\delta_3)$-stable. For $r=n-1,n$, assume
\begin{align*}
 \E[\delta_1(r,Z)]+\E[\delta_2(r,Z)]<\infty,\qquad \E[|L(T_{0_r}(X),Y)|]<\infty,
\end{align*}
where $0_r$ is the zero dataset of size $r$. In addition, $f$ is locally Lipschitz and, for constants $c_{n-1},d_{n-1}\geq0$,
\begin{equation}\label{eq:expected-loss-gradient}
 \|\nabla f(z)\|\leq c_{n-1}+d_{n-1}\|z\|,\qquad \text{for Lebesgue-almost every }z\in\R^{k+m}.
\end{equation}
\end{assumption}

\subsection{The Lipschitz case}

We first consider the case where the fitted loss is Lipschitz as a function of the data. In this case the growth term $\delta_2$ vanishes, and the concentration terms below have Gaussian tails.
\begin{theorem}[Lipschitz case]\label{thm:simplified}
Suppose that \cref{assumption:standing} holds, $\delta_2(r,z)=0$ for $r=n-1,n$, and $f$ is $c_{n-1}$-Lipschitz. For $\eps>3\delta_3(n)$, put $t=(\eps-\delta_3(n))/3$. Then
\begin{align*}
 \P\left ( L_{LOO}-\hat L_{LOO} > \eps \right )
 \leq n\E\left[e^{-\frac{t^2}{2\sigma^2(\mu)\delta_1^2(n-1,Z)}}\right] +e^{-\frac{nt^2}{2\sigma^2(\mu)c_{n-1}^2}}
 +e^{-\frac{t^2}{2\sigma^2(\mu)(\E[\delta_1(n,Z)])^2}}.
\end{align*}
If, in addition, $\delta_1(n-1,z)\leq a_{n-1}<\infty$ for $\mu$-almost every $z$, set
\begin{align*}
 s_n=a_{n-1}+\frac{c_{n-1}}{\sqrt n}+\E[\delta_1(n,Z)].
\end{align*}
Then, for every $t>0$,
\begin{equation}\label{eq:uniform-simple}
 \P(L_{{LOO}}-\hat L_{{LOO}}>\delta_3(n)+t)
 \leq e^{-\frac{t^2}{2\sigma^2(\mu)s_n^2}}.
\end{equation}
If $c_{n-1}$ is uniformly bounded, $\delta_1(r,z)\leq A/\sqrt r$ for $r=n-1,n$ and $\mu$-almost every $z$, and $\delta_3(n)\to0$, then for each fixed $\eps>0$ and all sufficiently large $n$,
\begin{equation}\label{eq:one-sided-simple}
 \P\left ( L_{{LOO}}-\hat L_{{LOO}} > \eps \right )
 \leq 3ne^{-C\eps^2 n}.
\end{equation}
Here $C>0$ depends only on $\sigma^2(\mu)$, $A$ and the uniform bound on $c_{n-1}$.
\end{theorem}

\begin{remark}\label{rem:errors}
		In \cref{thm:simplified}, the right-hand side contains three error terms which control the following three errors respectively. In the display below, recall explicitly that
		\[
			f_i(Z_i)=L(T_{D_{(-i)}}(X_i),Y_i),
			\qquad
			f_D(Z)=L(T_D(X),Y).
		\]
		\begin{align}\label{eq:decompose:error}
			\notag L_{{LOO}} - \hat L_{{LOO}}
			=&
			\frac{1}{n} \sum_{i=1}^n \left ( \E[f_i(Z_i) \mid Z_i] -f_i(Z_i)\right )
			\\
			\notag &+\frac{1}{n} \sum_{i=1}^n \left(\E[f_i(Z_i)]-\E[f_i(Z_i) \mid Z_i] \right )
			\\
			&+\frac{1}{n} \sum_{i=1}^n \left(\E[f_D(Z) \mid D]-\E[f_i(Z_i)]\right).
		\end{align}
		The first error measures the fluctuation of the fitted leave-one-out loss around its expectation. The second error measures the sampling fluctuation of the expected fitted loss evaluated at the held-out point $Z_i$. The third error measures the bias induced by deleting one observation, together with the remaining generalization-error contribution. Thus, the deviation level has to dominate the delete-one bias; this is the reason for the condition $\varepsilon > 3\delta_3(n)$.
	\end{remark}

\begin{remark}
	The proof also gives the opposite one-sided bound in \cref{eq:uniform-simple}. Adding the two tail probabilities yields, for every $t>0$,
	\begin{equation*}
		\P(|L_{{LOO}}-\hat L_{{LOO}}|>\delta_3(n)+t)
		\leq 2e^{-\frac{t^2}{2\sigma^2(\mu)s_n^2}}.
	\end{equation*}
	Under the rate assumptions leading to \cref{eq:one-sided-simple}, we also have
	\begin{align*}
		\P\left ( |L_{{LOO}}-\hat L_{{LOO}}| > \eps \right )
		\leq 6ne^{-C\eps^2 n}.
	\end{align*}
	The same applies to \cref{thm:main}; in \cref{thm:data-dependent}, the exceptional-set probability is counted only once. Note that while \cref{eq:one-sided-simple} measures the probability of under-estimating the true risk $L_{{LOO}}$, the opposite one-sided concentration
	\begin{align*}
		\P\left (  \hat L_{{LOO}} - L_{{LOO}} > \eps \right )
	\end{align*}
	measures the probability of over-estimating the true risk $L_{{LOO}}$.
\end{remark}

\subsection{The quadratic case}

\cref{thm:simplified} is a special case of the main theorem. In the general case the gradient of the fitted loss may have linear growth in the data norm, and the resulting estimates have both Gaussian and sub-exponential regimes. To state the theorem compactly, we first introduce the rate functions corresponding to the three errors in \cref{rem:errors}.

For $r\geq1$, $t>0$ and $a,b\geq0$, define
\begin{equation}\label{eq:theta-general}
 \Theta_r(t;a,b)=\exp\left[-\min\left\{
 \frac{t^2}{8\sigma^2(\mu)(a^2+2b^2r\E[\|Z\|^2])},\,
 \frac{t}{8\sigma^2(\mu)b}\right\}\right].
\end{equation}
A term with zero denominator is interpreted as $+\infty$ inside the minimum. Thus $\Theta_r(t;0,0)=0$. Let
\begin{align}
 \theta_1(n,t,z)&=\Theta_{n-1}(t;\delta_1(n-1,z),\delta_2(n-1,z)),\label{eq:theta1-1}\\
 \theta_2(n,t)&=\Theta_n\left(t;\frac{c_{n-1}}{\sqrt n},\frac{d_{n-1}}n\right),\label{eq:theta2-1}\\
 \theta_3(n,t)&=\Theta_n(t;\E[\delta_1(n,Z)],\E[\delta_2(n,Z)]).\label{eq:theta3-1}
\end{align}
\begin{theorem}\label{thm:main}
Suppose that \cref{assumption:standing} holds. Then for $\eps>3\delta_3(n)$ we have
\begin{align}\label{eq:main-bound}
 \P\left ( L_{{LOO}}-\hat L_{{LOO}} > \eps \right )
 \leq n\E[\theta_1(n,\eps/3,Z)]+\theta_2(n,\eps/3)+\theta_3(n,\eps/3-\delta_3(n)).
\end{align}
If $\delta_1(n-1,z)\leq a_{n-1}<\infty$ and $\delta_2(n-1,z)\leq b_{n-1}<\infty$ for $\mu$-almost every $z$, then also
\begin{equation}\label{eq:uniform-main}
 \P(L_{{LOO}}-\hat L_{{LOO}}>\eps)
 \leq\Theta_{n-1}(\eps/3;a_{n-1},b_{n-1})+\theta_2(n,\eps/3)+\theta_3(n,\eps/3-\delta_3(n)).
\end{equation}
Both bounds hold with $\hat L_{{LOO}}-L_{{LOO}}$ in place of $L_{{LOO}}-\hat L_{{LOO}}$.
\end{theorem}

\begin{remark}
The condition \cref{eq:expected-loss-gradient} permits quadratic expected losses. For example, $f(z)=(z-m)^2+v/r$ satisfies $|f'(z)|\leq2|m|+2|z|$. More generally, integration along a segment gives
\begin{align*}
 |f(z)-f(z')|\leq\|z-z'\|\left[c+\frac d2(\|z\|+\|z'\|)\right]
\end{align*}
when $\|\nabla f(z)\|\leq c+d\|z\|$.
\end{remark}

\begin{remark}
	In the Lipschitz case, i.e.~if $\delta_2(n,z)\equiv 0$ and $d_{n-1}=0$, both $\theta_1,\theta_2$ have exponents quadratic in $t$, i.e.~the sub-Gaussian terms dominate. For fixed $t>0$, bounded $c_{n-1},d_{n-1}$, and uniform bounds $\delta_1(r,z)\leq A/\sqrt r$, $\delta_2(r,z)\leq A/r$, \cref{eq:uniform-main} gives a bound of order $e^{-c_t n}$. Together with $\delta_3(n)\to0$, this gives exponential concentration at every fixed positive deviation. In other words, if the fitted loss is genuinely Lipschitz in the data, so that $\delta_2=0$, and $d_{n-1}=0$, the resulting bounds are of Gaussian or sub-Gaussian type. If a non-zero $\delta_2$ is needed to allow linear growth in $\|D\|$, then the same estimate has the usual two-regime form: sub-Gaussian for moderate deviations, and sub-exponential for larger deviations.
\end{remark}

\begin{remark}
		If we assume that the loss function $L$ is Lipschitz in its first component, i.e.
	\begin{align*}
		|L(A_1,B)-L(A_2,B)| \leq C\|A_1-A_2\|,
	\end{align*}
	then the Lipschitz condition on the estimator $T$ implies Lipschitz condition, i.e.~$\delta_2(n,z) \equiv 0$ in \cref{assumption:standing}.
\end{remark}

For $2\leq p<\infty$ and $r\geq1$, put
\begin{align*}
 Q_{r,p}=\sqrt{r\E[\|Z\|^2]+\sigma^2(\mu)(p-2)}.
\end{align*}
\begin{corollary}[$L^p$ bound]\label{cor:loo-moments}
Suppose that \cref{assumption:standing} holds and $\delta_1(n-1,Z),\delta_2(n-1,Z)\in L^p(\mu)$ for some $2\leq p<\infty$. Set
\begin{equation}\label{eq:loo-moment-scale}
\begin{aligned}
 B_{n,p}=\sqrt{\sigma^2(\mu)(p-1)}\biggl[&\|\delta_1(n-1,Z)\|_{L^p(\mu)}
 +Q_{n-1,p}\|\delta_2(n-1,Z)\|_{L^p(\mu)}\\
 &+\frac{c_{n-1}}{\sqrt n}+\frac{d_{n-1}}nQ_{n,p}
 +\E[\delta_1(n,Z)]+\E[\delta_2(n,Z)]Q_{n,p}\biggr].
\end{aligned}
\end{equation}
Then $H=L_{{LOO}}-\hat L_{{LOO}}$ satisfies
\begin{equation}\label{eq:loo-moment}
 \|H-\E[H]\|_{L^p(\mu^{\otimes n})}\leq B_{n,p}.
\end{equation}
\end{corollary}
Since $|\E[H]|\leq\delta_3(n)$, the uncentered error satisfies $\|H\|_{L^p(\mu^{\otimes n})}\leq\delta_3(n)+B_{n,p}$. At $p=2$, the corollary gives $\V[H]\leq B_{n,2}^2$.

\subsection{Extensions beyond quadratic: conditioned concentration}

In many cases of interest the rates $\delta_1(n,z)$ and $\delta_2(n,z)$ may depend on the data $D$ as well. The preceding theorem then cannot be applied globally. The next result covers this situation by restricting the data to a large set $K$ on which the required Lipschitz control of the LOO error is available.
Let us consider an example to highlight why such a conditioned statement is needed.
\begin{example}\label{example:linear:bad}
Consider the case of simple linear regression, i.e.~for a dataset $D$ of length $n$ we let $T_D(x)$ be the estimated linear function and let the loss be the standard $L^1$ loss. We will see in \cref{sec:examples} that $L(T_D(x),y)$ does not satisfy \cref{assumption:standing} on the space of all possible data configurations. To understand the issue without going too much into detail, we note that when the $X$ coordinates of the data are tightly grouped together, then the slope of $T_D(x)$ will change very rapidly and thus violate \cref{eq:gradient-condition} unless we let $\delta_1$ and $\delta_2$ also depend on $D$. We therefore control the LOO error on a set where the centered design variance stays away from zero.
\end{example}

The technical input is that, while the log-Sobolev constant $\sigma^2(\mu)$ is not stable under restrictions, the sub-Gaussian constant $\sigma_{SG}^2(\mu)$ is stable in a logarithmic sense. Specifically, by~\cite[Theorem 1.1]{Bobkov-et-al}, for a set $A$ such that $\mu(A) > 0$, the restricted measure
\begin{align*}
	\mu_A(B) = \frac{\mu(A \cap B)}{\mu(A)}
\end{align*}
satisfies
\begin{align}\label{eq:BNT}
	\sigma_{SG}^2(\mu_A) \leq c \log\left ( \frac{e}{\mu(A)} \right ) \sigma_{SG}^2(\mu),
\end{align}
{where $c = 3 \cdot 2^{12} e^2$.} We write $c_{\mathrm{BNT}}=c$ below.

Conditioning here means restricting the law of the whole dataset to the event $D\in K$. When $\P(D\in K)>0$, write $\P_K(A)=\P(A\mid D\in K)$ and $\E_K[H]=\E[H\mid D\in K]$. The observations become dependent under this law, so we assume a Lipschitz bound for the complete LOO error on $K$. The conclusion below is still a probability under the original law: we add the probability of $D\notin K$ to the concentration bound under $\P_K$.
\begin{theorem}\label{thm:data-dependent}
Let $D\sim\mu^{\otimes n}$, with log-Sobolev constant at most $\sigma^2(\mu)$ and $\E[\|Z\|^2]<\infty$. Put $H(D)=L_{{LOO}}-\hat L_{{LOO}}$, and let $B,\delta_3(n)\geq0$. Suppose that $H$ satisfies
\begin{align*}
 \|H\|_{L^2(\mu^{\otimes n})}\leq B,\qquad |\E[H]|\leq\delta_3(n).
\end{align*}
Let $K\subset\R^{(k+m)n}$ be a Borel set with $\P(D\in K)>0$, and let $L_K\geq0$ satisfy
\begin{equation}\label{eq:data-dependent-condition}
 |H(D)-H(D')|\leq L_K\|D-D'\|,\qquad D,D'\in K.
\end{equation}
Let $p=\P(D\notin K)$. Set
\begin{align*}
 C_K=c_{\mathrm{BNT}}\sigma^2(\mu)\log\frac{e}{1-p},\qquad
 \eta_K=B\sqrt{\frac{p}{1-p}}.
\end{align*}
For any $\eps>\delta_3(n)+\eta_K$, we have
\begin{equation}\label{eq:conditioned-bound}
 \P\left(L_{{LOO}}-\hat L_{{LOO}}>\eps\right)
 \leq\P(D\notin K)+\exp\left[-\frac{(\eps-\delta_3(n)-\eta_K)^2}{2C_KL_K^2}\right].
\end{equation}
\end{theorem}

In practice \cref{thm:data-dependent} can be applied by choosing a set $K$ growing in $n$. Then possible growth in $L_K$ in terms of the size of $K(n)$ makes the exponential rate worse, while $\mathbb{P}(D \notin K)$ can decay exponentially for suitable choices of $K(n)$. Hence, one can optimize by choosing suitable $K = K(n)$, depending on the situation.

{
	\begin{remark}
		The log-Sobolev inequality does not in general hold for bounded data, and as such our result focuses on extending concentration bounds beyond known results for bounded data. In our case, the tail behavior is contained in the log-Sobolev constant and in the quantity $\P(D \notin K)$, see also \cref{sec:conclusion}.
	\end{remark}
}

\section{Examples}\label{sec:examples}

\subsection{The empirical mean}
Consider the empirical mean $\overline X_n$. In this case we can consider $f_D(z) = L(\overline{X}_n,z) := |\overline X_n - z|$. Then, by the triangle inequality and Hölder inequality, we get
\begin{align*}
	|f_{D_1}(z)-f_{D_2}(z)|\leq |\overline{X}_{n,1}-\overline{X}_{n,2}|\leq \frac{\|D_1-D_2\|}{\sqrt{n}}
\end{align*}
which gives us $\|\nabla_D f_D(z)\| \leq \frac{1}{\sqrt{n}}$.
With $f_i(z) = L(T_{D_{(-i)}}(x),y)$ we also have
\begin{align*}
	|\E[f_i(Z_i)] - \E[f_D(Z)]| \leq \E|\overline{X}_n - \overline{X}_{n-1}| \leq 2 \frac{\E|X_1|}{n-1}
\end{align*}
allowing us to choose $\delta_3(n) = \frac{C}{n-1}$.
The expected fitted loss $z\mapsto\E[|\overline X_{n-1}-z|]$ is $1$-Lipschitz.
Hence, we obtain, by \cref{thm:simplified},
\begin{align*}
	\P\left ( L_{{LOO}}-\hat L_{{LOO}} > \eps \right ) \leq 3ne^{-C\eps^2 n}
\end{align*}
for large enough $n$. We note that by using permutation symmetry of $\overline{X}_n$, one can improve the upper bound to be $ c_1e^{-c_2\eps^2 n}$.

\subsection{\texorpdfstring{$Z$}{Z} dependent Lipschitz constant}
Let $L(x,y) = (x-y)^2$. Then for $f_D(z) = L(T_D(x),y) = (T_D(x)-y)^2$ from which it follows that
\begin{align*}
	\|\nabla_D f_D(z)\| = \|\nabla_D L(T_D(x),y)\| \leq 2\|\nabla_D T_D(x)\||y-T_D(x)|.
\end{align*}
If now $|T_D(x)|\leq C\|x\|$ and $\|\nabla_D T_D(x)\| \leq \delta_1(n)$, it follows that
\begin{align*}
	\|\nabla_D f_D(z)\| \leq C\delta_1(n)(|y|+\|x\|):= \delta_1(n,z).
\end{align*}

This provides an example where $\delta_1(n,z)$ depends on $z$ as well, highlighting the fact that in general the $z$-dependence should be taken into account.

For the empirical mean of scalar observations with mean $m$ and variance $v$, we have
\begin{align*}
 f_{D_r}(z)=(\bar Z_r-z)^2,\qquad
 \|\nabla_{D_r}f_{D_r}(z)\|\leq\frac{2|z|}{\sqrt r}+\frac2r\|D_r\|.
\end{align*}
Also,
\begin{align*}
 \E[(\bar Z_r-z)^2]=(z-m)^2+\frac vr,\qquad
 \left|\frac{d}{dz}\E[(\bar Z_r-z)^2]\right|\leq2|m|+2|z|.
\end{align*}
Thus $\delta_1(r,z)=2|z|/\sqrt r$, $\delta_2(r,z)=2/r$, $c_r=2|m|$ and $d_r=2$. The mean risks at training sizes $n$ and $n-1$ differ by $v/[n(n-1)]$, so we may take $\delta_3(n)=v/[n(n-1)]$. The assumptions of \cref{thm:main} therefore hold for the quadratic loss.

\subsection{Kernel density estimation}\label{ssec:density}
Consider kernel density estimation with given kernel $K$ in $\mathbb{R}$, for simplicity. Here $g$ is our true density function and $\hat g$ is our estimate given by
\begin{align*}
	\hat{g}(x) = \frac{1}{n}\sum_{j=1}^n K_h(x-x_j) = \frac{1}{nh}\sum_{j=1}^n K\left(\frac{x-x_j}{h}\right),
\end{align*}
where $h$ is a suitably chosen parameter (typically one chooses $h$ to depend on $n$ as well).
Assume that $K$ is nonnegative, bounded and continuously differentiable, with $\int K=1$ and $\|K'\|_\infty<\infty$. For a dataset $D_r$ of length $r$, write $\hat g_{D_r}(z)=r^{-1}\sum_{j=1}^r K_h(z-Z_j)$.
A common loss is the integrated $L^2$ loss, i.e.
\begin{align*}
	L(g,\hat g) = \int_{\mathbb{R}} (g(x)-\hat g(x))^2dx.
\end{align*}
The risk is then given by $R(g,\hat g) = \E[L(g,\hat g)]$.
Expanding the expression for the loss, we get
\begin{align*}
	L(g,\hat g) = \int_{\mathbb{R}} g^2(x)dx - 2\int_{\mathbb{R}} g(x) \hat g(x)dx + \int_{\mathbb{R}} \hat g^2(x)dx.
\end{align*}
In order to estimate the above, one typically omits the deterministic parts and instead focus on estimating the random parts, i.e.~the last two integrals. The last term is computable, while the second term equals  $2\E[\hat g(X)\mid D]$ which we need to estimate. One way to do this is to use leave-one-out estimation.
In this case we have $f_i(z) = \hat g_{D_{(-i)}}(z)$. Take $T_{D_r}(x)=\hat g_{D_r}(x)$ and $L(a,y)=|a|$. Since $K\geq0$, the fitted density is nonnegative. We obtain
\begin{align*}
 |\nabla_{x_k}\hat g_{D_{(-i)}}(z)|=\frac{1}{n-1}|K_h'(z-x_k)|\leq\frac{\|K'\|_\infty}{(n-1)h^2},
\end{align*}
leading to
\begin{align*}
 \|\nabla_{D_{(-i)}}\hat g_{D_{(-i)}}(z)\|\leq\frac{\|K'\|_\infty}{h^2\sqrt{n-1}}.
\end{align*}
That is, we have $\delta_1(r,z)=\|K'\|_\infty/[h^2\sqrt r]$ and $\delta_2(r,z)\equiv0$. The expected fitted loss is $\E[K_h(z-Z)]$, independently of $r$, and is $\|K'\|_\infty/h^2$-Lipschitz.
Furthermore, now $\E[f_i(Z_i)]=\E[f_D(Z)]$, so $\delta_3(n)=0$. Hence, we may apply \cref{thm:main} with $\delta_1(r,z)=\|K'\|_\infty/[h^2\sqrt r]$. For fixed $h$, \cref{thm:simplified} gives exponential concentration at every fixed positive deviation.

For a bandwidth sequence $h_n$, suppose the full and deleted fits both use $h_n$. Denote by $T_r^{(n)}$ the estimator trained on $r$ observations with bandwidth $h_n$. At each fixed $n$, \cref{thm:simplified} applies to $T_n^{(n)}$ and $T_{n-1}^{(n)}$, with zero delete-one bias. A sufficient condition for the bound to tend to zero at a fixed deviation is
\begin{align*}
 nh_n^4 \to \infty.
\end{align*}
If instead the full and deleted fits use $h_n$ and $h_{n-1}$, respectively, their bias is
\begin{align*}
 \delta_3(n)=\left|\E[K_{h_n}(Z-Z')]-\E[K_{h_{n-1}}(Z-Z')]\right|,
\end{align*}
where $Z,Z'$ are independent.

\subsection{Linear regression}\label{ssec:linear}
An example which requires \cref{thm:data-dependent} is simple linear regression. Consider the model $y = \alpha + \beta x + e$, where $x \sim N(0,1)$ is a random variable and $e \sim N(0,1)$ is independent of $x$. Then the least squares estimators for parameters $\alpha$ and $\beta$ are given by
\begin{align*}
	\hat \alpha &= \overline{y} - \hat\beta \overline{x} \\
	\hat\beta &= \frac{\sum_{i=1}^n (x_i-\overline{x})(y_i-\overline{y})}{\sum_{i=1}^n (x_i-\overline{x})^2}.
\end{align*}
Consider also the $L^1$-loss, i.e.~$L(x,y) = |x-y|$.
The denominator can become arbitrarily small, so the global stability assumption is unavailable. We apply \cref{thm:data-dependent} in the independent coordinates $(X_i,e_i)$ and estimate the difference of the LOO errors for two datasets in a suitable set.

We will now show how to apply \cref{thm:data-dependent} to this problem and thus show that there exist constants $C,c$ such that for a fixed $\eps>0$ and all sufficiently large $n$ we get
\begin{equation}\label{eq:linear:poop}
	\P\left ( L_{{LOO}} - \hat L_{{LOO}} > \eps \right ) \leq C n e^{-cn}.
\end{equation}
The constants can be chosen independently of $\alpha$ and $\beta$.

\paragraph{Step 1: Change of variables.} We work with the independent Gaussian coordinates $D=((X_i,e_i))_{i=1}^n$, reconstructing the responses as $Y_i=\alpha+\beta X_i+e_i$. For a dataset $D=((x_j,e_j))_{j=1}^r$, we have
\begin{align*}
 \hat\beta_D-\beta
 =\frac{\sum_j(x_j-\bar x)(e_j-\bar e)}{\sum_j(x_j-\bar x)^2},\qquad
 T_D(u)-\alpha-\beta u=\bar e+(\hat\beta_D-\beta)(u-\bar x).
\end{align*}
With probability $1$ the above is well defined.
Note also that here $\sigma^2(\mu)=1$ and $\E[X_i^2+e_i^2]=2$.

\paragraph{Step 2: The set $K$.} Let
\begin{equation}\label{eq:ols-good}
 K=\left\{D=((x_i,e_i))_{i=1}^n:\|D\|\leq3\sqrt n,\ \sum_i(x_i-\bar x)^2\geq\frac n2,\ \max_ix_i^2\leq\frac n{64}\right\}.
\end{equation}
For $D\in K$, the identity for deleting one observation gives
\begin{align*}
 \sum_{j\ne i}(x_j-\bar x_{(-i)})^2
 =\sum_j(x_j-\bar x)^2-\frac n{n-1}(x_i-\bar x)^2
 \geq\frac n2-4x_i^2-4\bar x^2
 \geq\frac n2-\frac n{16}-36\geq\frac n4,
\end{align*}
provided that $n$ is large enough ($n\geq192$).
Here we used $n/(n-1)\leq2$, $(x_i-\bar x)^2\leq2x_i^2+2\bar x^2$, and $\bar x^2\leq\|D\|^2/n\leq9$. Deleting coordinates cannot increase the Euclidean norm, so every full or deleted dataset also has norm at most $3\sqrt n$.

There are numerical constants $C,c>0$ such that, with $p=\P(D\notin K)$,
\begin{equation}\label{eq:ols-bad}
 p\leq C n e^{-cn}.
\end{equation}
Indeed, $\|D\|^2$ has the $\chi^2_{2n}$ law and $\sum_i(X_i-\bar X_n)^2$ the $\chi^2_{n-1}$ law. Using $\E[e^{t\chi_k^2}]=(1-2t)^{-k/2}$ and Markov's inequality with $t=1/4$ gives
\begin{align*}
 \P(\|D\|^2>9n)\leq e^{-9n/4}2^n,
 \qquad
 \P\left(\sum_i(X_i-\bar X_n)^2<n/2\right)
 \leq e^{n/8}(3/2)^{-(n-1)/2}.
\end{align*}
Both terms decay exponentially in $n$. Also, the Gaussian tail bound and a union bound give $\P(\max_iX_i^2>n/64)\leq2n e^{-n/128}$. Adding these three probabilities proves \cref{eq:ols-bad}.

\paragraph{Step 3: Lipschitz continuity on $K$.} Let $D,D'$ have the same size $r\in\{n-1,n\}$, norms at most $3\sqrt n$, and satisfy $\sum_{j=1}^r(x_j-\bar x)^2\geq n/4$ and $\sum_{j=1}^r(x'_j-\bar x')^2\geq n/4$. Centering is an orthogonal contraction, so the numerators and denominators in the formula for $\hat\beta_D-\beta$ satisfy
\begin{align*}
 \left|\sum_j(x_j-\bar x)(e_j-\bar e)
 -\sum_j(x'_j-\bar x')(e'_j-\bar e')\right|
 \leq6\sqrt n\|D-D'\|,\\
 \left|\sum_j(x_j-\bar x)^2-\sum_j(x'_j-\bar x')^2\right|
 \leq6\sqrt n\|D-D'\|.
\end{align*}
Both denominators are at least $n/4$, while the second numerator has absolute value at most $9n$. Subtracting the two quotients therefore gives
\begin{align*}
 |\hat\beta_D-\hat\beta_{D'}|
 \leq\frac{6\sqrt n}{n/4}\|D-D'\|
 +\frac{9n\,6\sqrt n}{(n/4)^2}\|D-D'\|
 =\frac{888}{\sqrt n}\|D-D'\|.
\end{align*}
The Cauchy--Schwarz inequality also gives $|\hat\beta_D-\beta|\leq6$ and $|\hat\beta_{D'}-\beta|\leq6$. Moreover, the sample means of the covariates and noises have absolute values at most $5$ and change by at most $2\|D-D'\|/\sqrt n$. Substituting these estimates into the formula for $T_D(u)-\alpha-\beta u$ gives, with a numerical constant $C$,
\begin{equation}\label{eq:ols-pairwise}
 |T_D(u)-T_{D'}(u)|\leq\frac{C(1+|u|)}{\sqrt n}\|D-D'\|.
\end{equation}
For $D,D'\in K$, integrating \cref{eq:ols-pairwise} against an independent standard normal test covariate gives
\begin{align*}
 |L_{{LOO}}(D)-L_{{LOO}}(D')|\leq\frac{2C}{\sqrt n}\|D-D'\|.
\end{align*}
In each validation summand, we first change the training dataset and then the held-out observation. The reverse triangle inequality and $|\hat\beta_{D'_{(-i)}}-\beta|\leq6$ give
\begin{align*}
 \Bigl||e_i-(T_{D_{(-i)}}(x_i)-\alpha-\beta x_i)|
 -|e_i'-(T_{D'_{(-i)}}(x_i')-\alpha-\beta x_i')|\Bigr|\\
 \leq |e_i-e_i'|+6|x_i-x_i'|
 +\frac{C(1+|x_i|)}{\sqrt n}\|D_{(-i)}-D'_{(-i)}\|.
\end{align*}
Since $\|D\|\leq3\sqrt n$, Cauchy--Schwarz gives $n^{-1}\sum_i|x_i|\leq3$. It also gives $n^{-1}\sum_i|e_i-e_i'|\leq\|D-D'\|/\sqrt n$ and the analogous estimate for the covariates. Finally, deleting one coordinate pair cannot increase the distance, so $\|D_{(-i)}-D'_{(-i)}\|\leq\|D-D'\|$. Averaging the preceding inequality over $i$ therefore gives
\begin{align*}
 |\hat L_{{LOO}}(D)-\hat L_{{LOO}}(D')|
 \leq\frac{4C+7}{\sqrt n}\|D-D'\|.
\end{align*}
Thus $H$ is $C/\sqrt n$-Lipschitz on $K$, after increasing $C$.

\paragraph{Step 4: Bounds on the moments.} For a Gaussian sample of size $r>3$ and an independent test covariate $X$, the sample mean of the noises and $\hat\beta_D-\beta$ are conditionally centered Gaussian variables with variances $1/r$ and $1/\sum_j(X_j-\bar X_r)^2$, and covariance zero. Hence
\begin{equation}\label{eq:ols-mse}
 \E[(T_D(X)-\alpha-\beta X)^2]
 =\frac1r+\E\left[\frac{1+\bar X_r^2}{\sum_j(X_j-\bar X_r)^2}\right]
 =\frac1r+\frac{1+1/r}{r-3}.
\end{equation}
The last equality uses independence of $\bar X_r$ and $\sum_{j=1}^r(X_j-\bar X_r)^2$, $\E[\bar X_r^2]=1/r$, and $\E[1/\chi^2_{r-1}]=1/(r-3)$, obtained by integrating its gamma density. The reverse triangle inequality and Cauchy--Schwarz at sample sizes $n$ and $n-1$ yield
\begin{align*}
 |\E[H]|\leq
 \sqrt{\frac1n+\frac{1+1/n}{n-3}}
 +\sqrt{\frac1{n-1}+\frac{1+1/(n-1)}{n-4}}
 \leq\frac C{\sqrt n}.
\end{align*}
For the second moments, conditional Jensen's inequality gives $L_{{LOO}}^2\leq\E_{X,e}[|Y-T_D(X)|^2\mid D]$. The cross term vanishes after averaging over the independent centered test noise $e$. Similarly, Jensen's inequality for the validation average gives $\hat L_{{LOO}}^2\leq n^{-1}\sum_i|Y_i-T_{D_{(-i)}}(X_i)|^2$. Each held-out pair is independent of its training sample, so \cref{eq:ols-mse} gives
\begin{equation*}
 \E[L_{{LOO}}^2]\leq1+\frac1n+\frac{1+1/n}{n-3},\qquad
 \E[\hat L_{{LOO}}^2]\leq1+\frac1{n-1}+\frac{1+1/(n-1)}{n-4},
\end{equation*}
and by $H^2\leq2L_{{LOO}}^2+2\hat L_{{LOO}}^2$ and \cref{eq:ols-mse} we get
$\E[H^2]\leq12$ for $n\geq6$.

\paragraph{Step 5: Application of \cref{thm:data-dependent}.} The set $K$ has positive Gaussian probability: it contains a neighborhood of the dataset with $x_i=(-1)^{i+1}$ and all noises zero. The assumptions of \cref{thm:data-dependent} are now verified with $B=\sqrt{12}$, $\delta_3(n)=C/\sqrt n$, and $L_K=C/\sqrt n$. Since $\sigma^2(\mu)=1$, the two-sided form of \cref{eq:conditioned-bound} gives
\begin{align*}
 \P\left(|H|>\frac C{\sqrt n}+\sqrt{\frac{12p}{1-p}}+t\right)
 \leq p+2\exp\left[-\frac{nt^2}{2c_{\mathrm{BNT}}C^2\log(e/(1-p))}\right].
\end{align*}
For a fixed $\eps>0$ and all sufficiently large $n$, we have $C/\sqrt n+\sqrt{12p/(1-p)}\leq\eps/2$. Set $t=\eps/2$. Since $p\to0$, the logarithm in the denominator stays bounded. Combining with \cref{eq:ols-bad} proves \cref{eq:linear:poop}.

\subsection{Truncation}
{Many standard estimators are almost stable. It seems that they still concentrate but are not covered by our result. We here provide a simple method to bound their training-data sensitivity via truncation.}
Consider the truncator $g_b (x) = \max(\min(x,b),-b)$ and consider $D_b = ((g_b(X_1),g_b(Y_1)),\ldots,(g_b(X_n),g_b(Y_n)))$. Now, if
\begin{align*}
	\|\nabla f[D]\| \leq C\frac{1}{n^{(q+1)/2}}\|D\|^q,
\end{align*}
then $\|D_b\|\leq b\sqrt{(k+m)n}$ and the coordinate truncation is $1$-Lipschitz. Absorbing the fixed dimension $k+m$ into $C$, we obtain
\begin{align*}
	\|\nabla_{D} f[D_b]\| \leq C \frac{b^q}{\sqrt{n}}.
\end{align*}

\subsection{Stabilized Linear regression}\label{ssec:stabilized_linear}
{As opposed to using our \cref{thm:data-dependent} and conditioning on the sets $K_n$ we can instead alter the linear regression method to stabilize it. It should be noted that in many standard cases, the effect of stabilization is small.}
Consider the stabilized estimators for parameters $\alpha$ and $\beta$ as
\begin{align*}
	\hat \alpha = \overline{y} - \hat\beta \overline{x}
\end{align*}
and let $g_b(y) = \max(\min(y,b),-b)$ be the $b$ level truncation. Let $\tilde y_i = g_b(y_i)$, and define the stabilized slope as
\begin{align*}
	\hat\beta = \frac{\sum_{i=1}^n (x_i-\overline{x})(\tilde y_i-\overline{\tilde y})}{\sum_{i=1}^n (x_i-\overline{x})^2 + n \delta},
\end{align*}
where $\delta > 0$ is a stabilization parameter.
Consider also the $L^1$-loss, i.e.~$L(x,y) = |x-y|$.
For the derivative and delete-one estimates, fix $b,\delta>0$ and put $U_j=g_b(Y_j)$. For a sample of size $r$, write
\begin{align*}
 a=\bar X_r,\qquad c=\bar U_r,\qquad v=\frac1r\sum_j(X_j-a)^2,\qquad s=\frac1r\sum_j(X_j-a)(U_j-c).
\end{align*}
The stabilized estimator is
\begin{align*}
 \hat\beta_r=\frac{s}{v+\delta},\qquad T_{D_r}(x)=\bar Y_r+\hat\beta_r(x-a).
\end{align*}
Cauchy--Schwarz gives $|s|\leq b\sqrt v$, and therefore
\begin{equation}\label{eq:ridge-slope}
 |\hat\beta_r|\leq\frac{b}{2\sqrt\delta}.
\end{equation}
At points where the truncation map $g_b$ is differentiable,
\begin{align*}
 \partial_{Y_j}\hat\beta_r=\frac{(X_j-a)g_b'(Y_j)}{r(v+\delta)},\qquad
 \partial_{X_j}\hat\beta_r=\frac{U_j-c-2\hat\beta_r(X_j-a)}{r(v+\delta)}.
\end{align*}
Using $\sum_j(U_j-c)^2\leq rb^2$ and $\sum_j(X_j-a)^2=rv$, we obtain
\begin{align*}
 \|\nabla_Y\hat\beta_r\|\leq\frac{\sqrt v}{\sqrt r(v+\delta)},\qquad
 \|\nabla_X\hat\beta_r\|\leq\frac{b+2|\hat\beta_r|\sqrt v}{\sqrt r(v+\delta)}.
\end{align*}
In particular, $\|\nabla\hat\beta_r\|\leq C_{b,\delta}/\sqrt r$. The maps are locally Lipschitz and piecewise continuously differentiable; the same bounds hold for their local slopes by integrating over segments and passing to the endpoint. It follows that
\begin{equation}\label{eq:ridge-gradient}
 \|\nabla_{D_r}T_{D_r}(x)\|\leq\frac{C_{b,\delta}(1+|x|)}{\sqrt r}+\frac{C_{b,\delta}}r\|D_r\|.
\end{equation}
The absolute fitted loss satisfies the same estimate. By \cref{eq:ridge-slope}, its expectation is globally Lipschitz in $(x,y)$, with constant at most $\sqrt{1+b^2/(4\delta)}$, uniformly in $r$. The fitted loss at $D_r=0$ is $|y|$.

We next estimate the delete-one bias. We compare datasets of lengths $r$ and $r+1$ by adding the observation $(X_*,Y_*)$. Denote the new empirical quantities by a subscript $+$ and put $d=X_*-a$, $e=g_b(Y_*)-c$. The exact updates are
\begin{align*}
 s_+=\frac r{r+1}s+\frac r{(r+1)^2}de,\qquad
 v_+=\frac r{r+1}v+\frac r{(r+1)^2}d^2.
\end{align*}
Using $|e|\leq2b$, \cref{eq:ridge-slope} and the denominator bound $v_++\delta\geq\delta$ gives
\begin{equation}\label{eq:ridge-delete}
 |\hat\beta_{r+1}-\hat\beta_r|
 \leq\frac{C_{b,\delta}}{r+1}\left(\sqrt v+|d|+v+d^2\right).
\end{equation}
Suppose $\E[|Y|]<\infty$ and $\E[|X|^4]<\infty$. Jensen's inequality gives $\E[v^2]\leq\E[X^4]$ and $\E[|d|^4]\leq16\E[X^4]$, uniformly in $r$. Thus the right-hand side of \cref{eq:ridge-delete} has $L^2(\mu^{\otimes(r+1)})$ norm at most $C/(r+1)$. For an independent test covariate $X$,
\begin{align*}
 T_{D_{r+1}}(X)-T_{D_r}(X)
 =\bar Y_{r+1}-\bar Y_r
 +(\hat\beta_{r+1}-\hat\beta_r)(X-a_+)
 +\hat\beta_r(a-a_+).
\end{align*}
The first term has expected absolute value at most $2\E[|Y|]/(r+1)$, the last at most $b\E[|X|]/[\sqrt\delta(r+1)]$, and the middle term has expected absolute value at most $C/(r+1)$ by Cauchy--Schwarz. The reverse triangle inequality for the absolute loss now gives $\delta_3(n)\leq C/n$.

We may now apply \cref{thm:main} with $\delta_1(r,(x,y))=\allowbreak C_{b,\delta}(1+|x|)/\sqrt r$, $\delta_2(r,(x,y))=C_{b,\delta}/r$, uniformly bounded $c_r$ and $d_r=0$. Under \cref{eq:lsi}, the required fourth moment follows from \cref{eq:norm-mgf}.

\subsection{Stabilized kernel regression}\label{ssec:kernel}

Consider a stabilized form of kernel regression. For simplicity, let $g(X) = \E[Y \mid X]$ be the regression function. Let $T_D$ be an estimate of $g(X)$ produced as the kernel regression estimate on the dataset $D$, and consider the $L^1$ prediction loss. That is, let
$
L(T_D(x),y) = |T_D(x)-y|
$
and define the stabilized kernel regression estimator as
\begin{align*}
	T_D(x) = \frac{\sum_{j=1}^{n} K_h(x-x_j) y_j}{\sum_{j=1}^n K_h(x-x_j) + n\delta}
\end{align*}
with a tuning parameter $h$.
Computing first the $x_i$ derivative gives
\begin{align*}
	\nabla_{x_i} T_D(x) & = \frac{-K_h'(x-x_i) y_i}{\sum_{j=1}^n K_h(x-x_j)+n\delta}+\frac{\sum_{j=1}^{n} K_h(x-x_j) y_j (K_h'(x-x_i))}{(\sum_{j=1}^n K_h(x-x_j)+n\delta)^2} \\
	&=\frac{K_h'(x-x_i)}{\sum_{j=1}^n K_h(x-x_j)+n\delta} \left (-y_i+T_D(x)\right ).
\end{align*}
If now $K$ is the Gaussian kernel and $h,\delta>0$ are fixed, we get
\begin{align*}
	\left|\frac{K_h'(x-x_i)}{\sum_{j=1}^n K_h(x-x_j)+n\delta}\right| = \frac{|x-x_i|}{h^2}\frac{K_h(x-x_i)}{\sum_{j=1}^n K_h(x-x_j)+n\delta}
\end{align*}
and hence, for a fixed $h$, we get
\begin{align*}
	\left|\frac{K_h'(x-x_i)}{\sum_{j=1}^n K_h(x-x_j)+n\delta}\right| \leq C(h,\delta) \frac{1}{n}.
\end{align*}
Since now $|T_D(x)|\leq C\|D\|/\sqrt{n}$, we obtain
\begin{align*}
	\left|\frac{1}{n} \left (-y_i+T_D(x)\right )\right| \leq \frac{C|y_i|}{n} + \frac{C\|D\|}{n\sqrt{n}}.
\end{align*}
Similarly, derivatives in $y_i$ directions satisfy
\begin{align*}
	|\nabla_{y_i} T_D(x)| & = \left|\frac{K_h(x-x_i)}{\sum_{j=1}^n K_h(x-x_j)+n\delta}\right|\leq \frac{C(h,\delta)}{n}.
\end{align*}
From this it follows that, using also $\sum_{j=1}^n |y_j| \leq \sqrt{n}\|D\|$, we have
\begin{align*}
	\|\nabla_D T_D(x)\|^2 \leq & C(h,\delta)\sum_{j=1}^n \left[\left(\frac{|y_j|}{n} + \frac{\|D\|}{n\sqrt{n}}\right)^2 + \frac{1}{n^2}\right] \\
	\leq & \frac{C}{n} + \frac{C\|D\|^2}{n^2}.
\end{align*}
Hence, we observe that, for all $x$,
\begin{equation}\label{eq:kernel-gradient}
	\|\nabla_{D_r} T_{D_r}(x)\| \leq \frac{C}{\sqrt{r}} + \frac{C\|D_r\|}{r}.
\end{equation}
As such, we may set $\delta_1(r,z)=C/\sqrt r$ and $\delta_2(r,z)=C/r$ for the absolute fitted loss. Here the same calculation applies at each sample size $r$, with denominator $r\delta$.

Put $W=\|K_h\|_\infty$, $W_1=\|K_h'\|_\infty$, and $w_j(x)=K_h(x-X_j)$.

For the delete-one comparison, put $A=\sum_{j\ne i}w_jY_j$ and $B=\sum_{j\ne i}w_j$. Subtracting the predictors at sample sizes $n$ and $n-1$ gives
\begin{equation}\label{eq:kernel-delete}
 T_{D_n}(x)-T_{D_{(-i)}}(x)
 =\frac{w_i(x)Y_i-(w_i(x)+\delta)T_{D_{(-i)}}(x)}{B+w_i(x)+n\delta}.
\end{equation}
Uniformly in $x$, $\E[|T_{D_r}(x)|]\leq W\E[|Y|]/\delta$. Consequently, for an independent test covariate,
\begin{align*}
 \E[|T_{D_n}(X)-T_{D_{(-i)}}(X)|]
 \leq\frac{W\E[|Y|]}{n\delta}\left(2+\frac W\delta\right).
\end{align*}
This bounds $\delta_3(n)$ for the absolute loss.

Differentiation in the test covariate gives
\begin{align*}
 |\partial_x T_{D_r}(x)|\leq\frac{W_1}{r\delta}\left(\sum_j|Y_j|+r|T_{D_r}(x)|\right).
\end{align*}
Integrating along the interval between two test covariates and taking expectations shows that $\E[|T_{D_r}(x)-y|]$ is globally Lipschitz, uniformly in $r$, with an $x$-Lipschitz constant at most $(W_1/\delta)(1+W/\delta)\E[|Y|]$ and a $y$-Lipschitz constant equal to one. Together with \cref{eq:kernel-gradient,eq:kernel-delete}, this verifies every assumption of \cref{thm:main}. Since \cref{eq:kernel-gradient} is uniform in the test point, \cref{eq:uniform-main} gives a two-sided bound of order $e^{-cn}$ for fixed $h,\delta$ and a fixed positive deviation.

\section{Numerical experiments}\label{sec:numerical}
{
	The aim of this section is to provide numerical evidence for the rate at which the level sets of the error $L_{{LOO}} - \hat L_{{LOO}}$ tend to zero w.r.t.~$n$. Namely, for a given $\epsilon > 0$ we are interested in
	\begin{align*}
		\P(|L_{{LOO}} - \hat L_{{LOO}}| > \eps).
	\end{align*}
	Furthermore, we provide simulations of the rate of the standard deviation, since if the probability above is bounded by $C e^{-c n \eps^2}$ for some constants $c,C$ and all $\eps > 0$, then the standard deviation would converge to $0$ with rate $1/\sqrt{n}$, which is what one would expect. The uniform bound in \cref{thm:simplified} and \cref{cor:loo-moments} also control fluctuations at smaller scales under their additional assumptions.
	If these upper bounds capture the correct exponential rate, then the probability above should appear approximately linear in $n$ on the log-scale.
	
	In the following we have provided experimental data on the case of linear regression, kernel density estimation and that of the stabilized kernel regression.
		In each experiment we used the sample sizes $n=10,20,\ldots,490$. For every sample size, we generated independent training data sets and computed the LOO estimate on the training data. The conditional risk $L_{{LOO}}$ was approximated by Monte Carlo simulation with sample size $10n$, using test samples generated independently of the training data. Each simulation batch contains $100$ repetitions for every value of $n$, and the plots were produced by aggregating the stored batches. For each threshold displayed below, the empirical exceedance probabilities were computed as frequencies over the aggregated repetitions. The displayed curves are kernel-smoothed empirical trends. The shaded bands are descriptive envelopes based on kernel-weighted dispersion in the Monte Carlo output; in the probability plots they describe variation of the batch estimates, without a confidence or prediction coverage interpretation. Finally, $5000$ samples were used to draw the residual plots.
	}

\subsection{Kernel density estimation}
In this context, referring back to \cref{ssec:density}, our goal is to estimate the density $g$ by using the dataset $D$. In order to evaluate the risk we need to compute
\begin{align*}
	L_{{LOO}} := \E[\hat g_{D}(Z) \mid D],
\end{align*}
where $Z$ is an independent copy of $Z_1 \in D$. The estimator of this is the leave-one-out estimator, i.e.
\begin{align*}
	\hat L_{{LOO}}:= \frac{1}{n} \sum_{i=1}^n f_i(Z_i),
\end{align*}
where $f_i(z) = \hat g_{D_{(-i)}}(z)$. In \cref{ssec:density} we verified that $\delta_1(n,Z) \approx \frac{1}{\sqrt{n}}$, which then gives us exponential concentration via \cref{thm:main}. For numerical illustration, let $X \sim \text{Uniform}(0,1)$ and consider the random variable $Y = \sin(10X)$. Now our goal is to estimate the density of $Y$, which we do by using the Gaussian kernel with a bandwidth $1/10$.

In \cref{fig:density} there are two plots. The first plot shows the value of $\hat L_{{LOO}} - L_{{LOO}}$ for different choices of $n$ (the $x$-axis), and indicates the expected $1/\sqrt{n}$ rate for the standard deviation. The second plot shows the empirical value of $\P(|\hat L_{{LOO}} - L_{{LOO}}| > .02)$ on the log-scale with respect to $n$. Its approximately linear trend is consistent with the exponential concentration predicted by \cref{thm:main}.

\begin{figure}
	\centering
	\includegraphics[width=0.3\columnwidth]{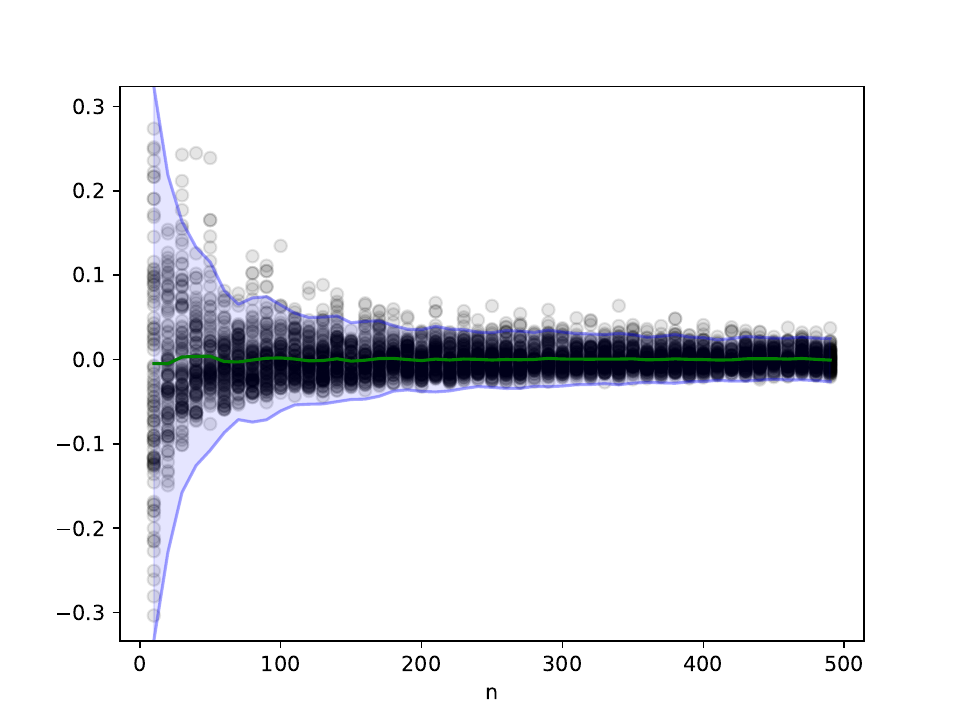}
	\includegraphics[width=0.3\columnwidth]{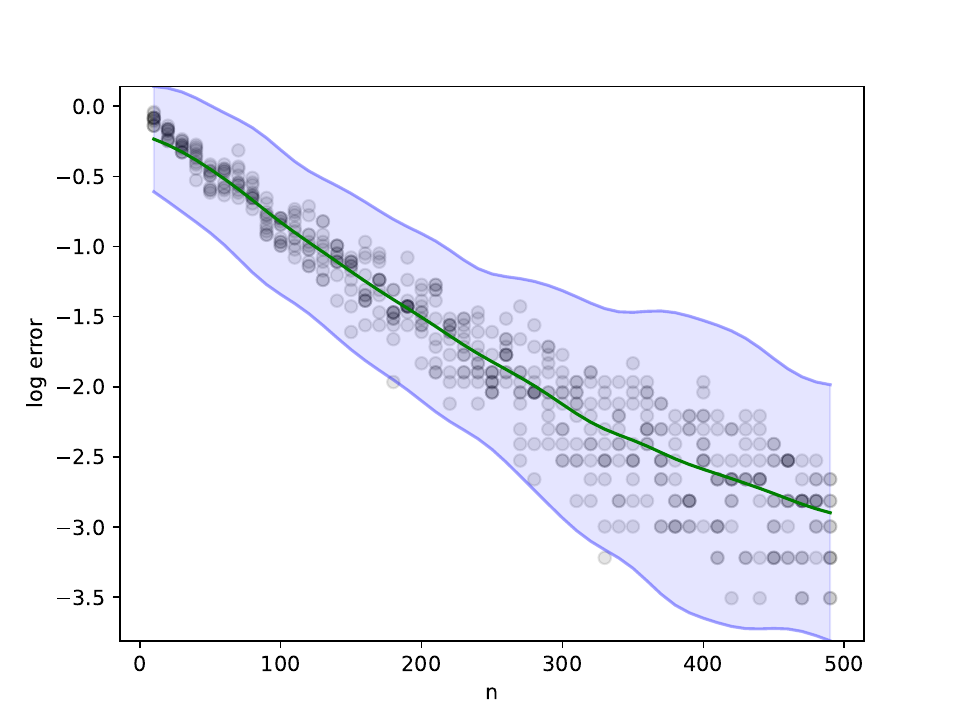}
	\caption{Leave one out cross-validation error using kernel density estimation with Gaussian kernel.}\label{fig:density}
\end{figure}

\subsection{Linear regression}
In this particular example we chose $X \sim N(0,1)$ and $Y \mid X \sim N(5X,1)$. Our fitted line on a dataset $D$ is $T_D$, and our loss is $L(x,y) = |x-y|$. In this case we measure the error
\begin{align*}
	\hat L_{{LOO}} := \frac{1}{n} \sum_{i=1}^n L(T_{D_{(-i)}}(X_i),Y_i),
\end{align*}
where
\begin{align*}
	L_{{LOO}} := \E[L(T_D(X),Y) \mid D]
\end{align*}
and $(X,Y)$ is an independent copy of $Z_1 \in D$. Even though we have to resort to conditioning in \cref{ssec:linear} and only obtain a weaker concentration through \cref{thm:data-dependent}, \cref{fig:linear} shows an exponential type concentration.
Furthermore, the second image in \cref{fig:linear} shows the empirical value of $\P(|\hat L_{{LOO}} - L_{{LOO}}| > .02)$ in a log-scale, again revealing relatively linear structure.

\begin{figure}
	\centering
	\includegraphics[width=0.3\columnwidth]{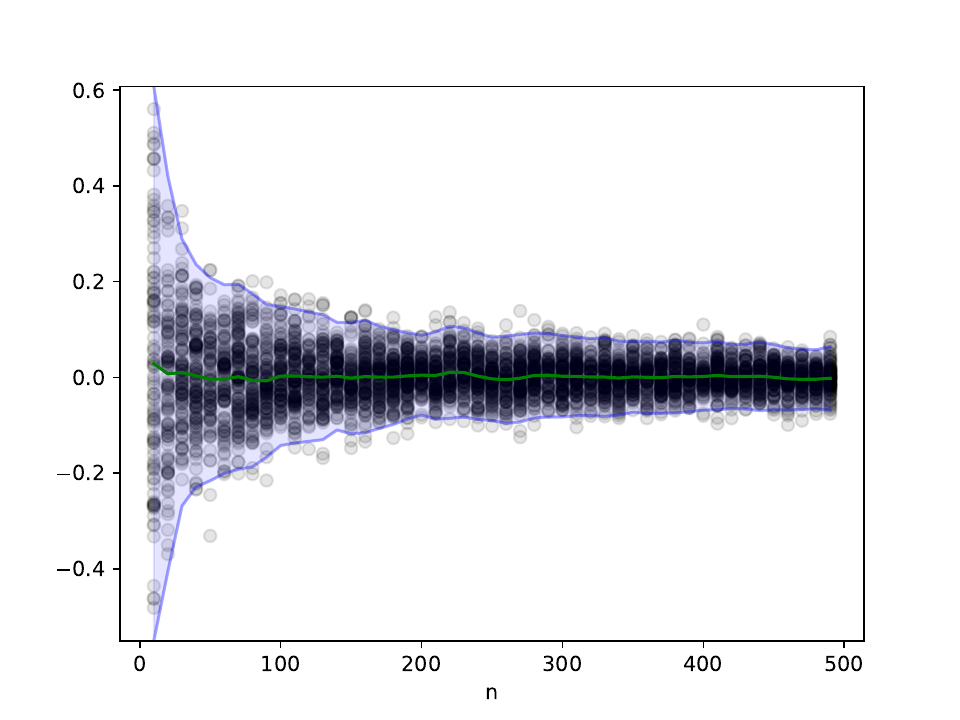}
	\includegraphics[width=0.3\columnwidth]{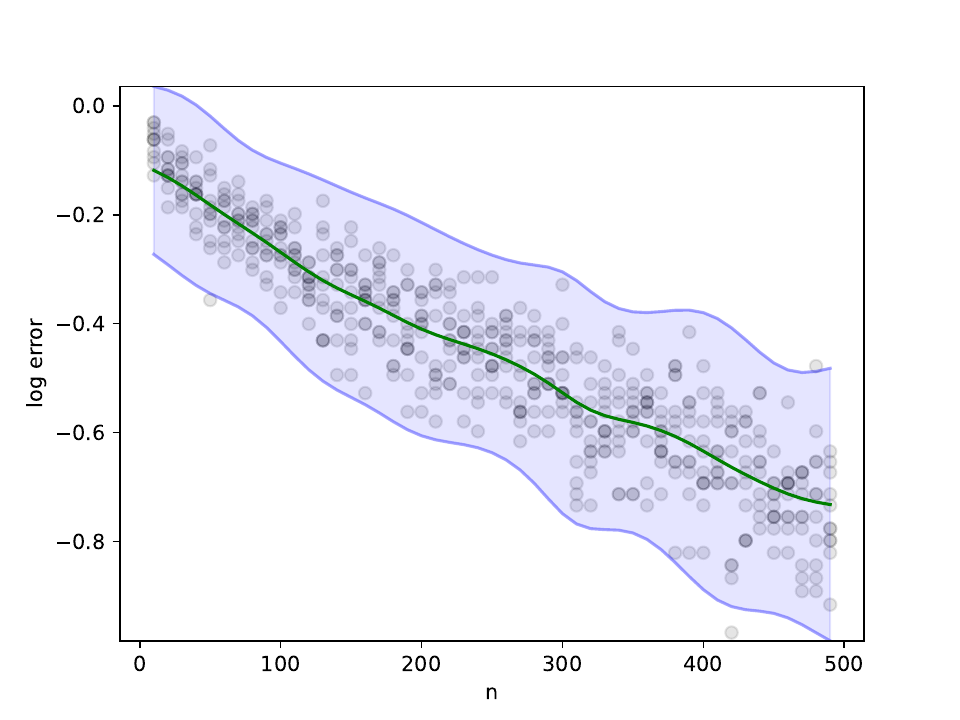}
	\caption{Leave one out cross-validation error using linear regression.}\label{fig:linear}
\end{figure}

\subsection{Stabilized kernel regression}
In this particular example we chose $X \sim N(0,1)$ and $Y \mid X \sim N(\sin(10X),1)$. Our fitted curve on a dataset $D$ is $T_D$, and our loss is $L(x,y) = |x-y|$.

For this simulation we used the unnormalized Gaussian weights $K_h(u)=e^{-u^2/(2h^2)}$, with bandwidth $h = 1/\sqrt{200}$ and stabilization parameter $\delta = 0.01$, keeping otherwise the setting as in the previous subsection. The result can be found in \cref{fig:kernel}. Again, we observe strong concentration as guaranteed by \cref{thm:main}.

\begin{figure}
	\centering
	\includegraphics[width=0.3\columnwidth]{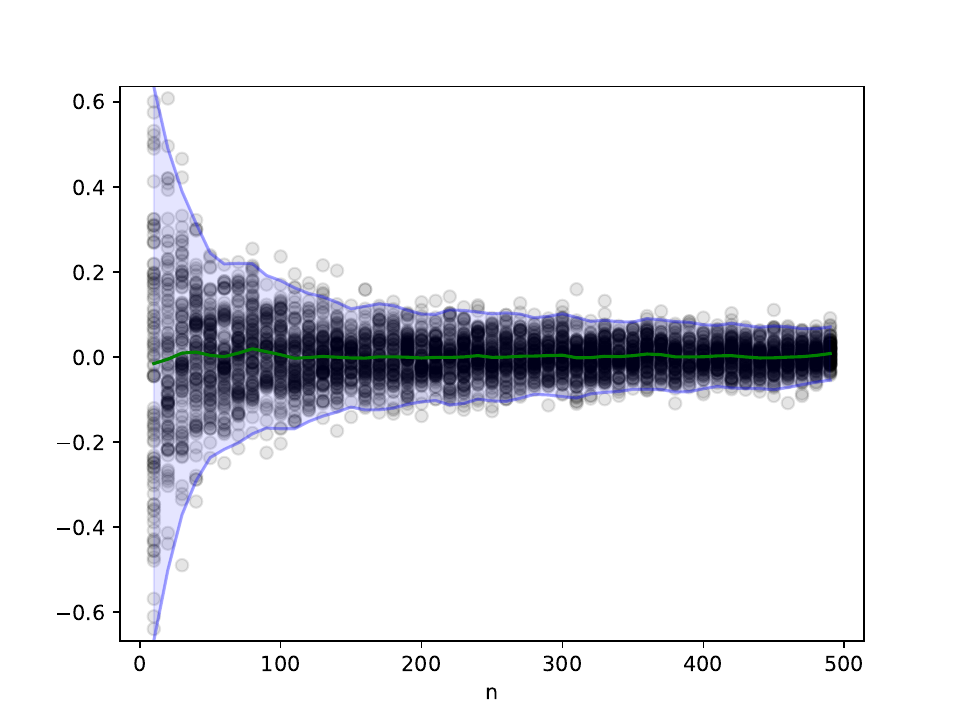}
	\includegraphics[width=0.3\columnwidth]{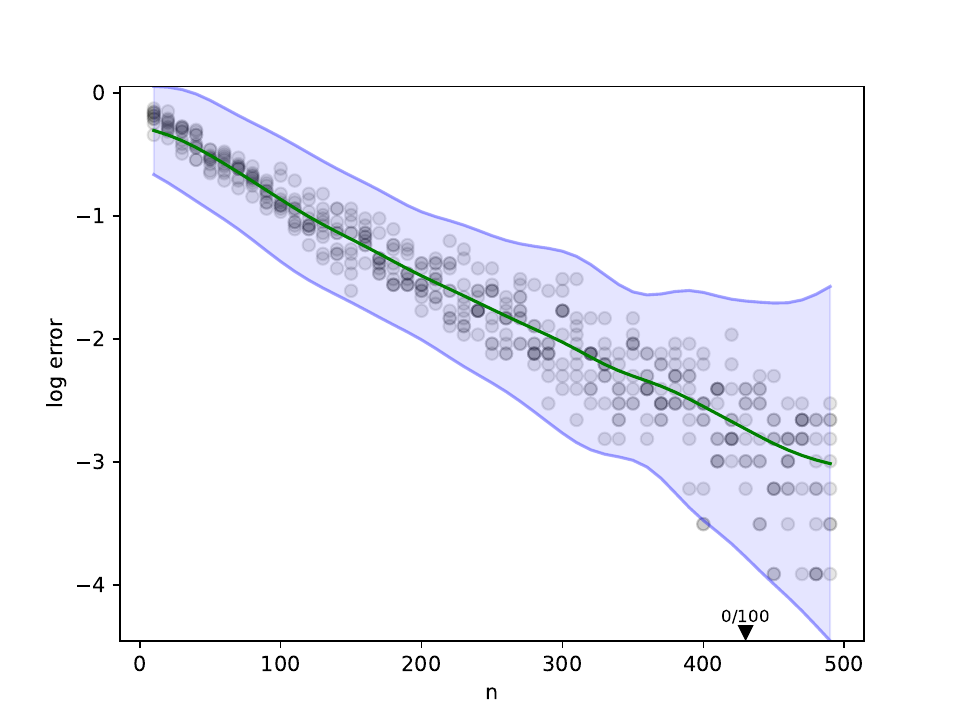}
	\caption{Leave one out cross-validation error using stabilized kernel regression.}\label{fig:kernel}
\end{figure}

\section{Proofs}\label{sec:proofs}

The first three results in this section are concentration inputs used later in the proof of the main LOO estimates. The estimate \cref{thm:subexp-concentration} is a standard sub-exponential average bound, while \cref{lemma:quadratic,cor:quadconc} give a self-contained log-Sobolev estimate for functions whose gradients have linear growth in the data. These results are close in spirit to the Hoeffding- and Bernstein-type inequalities of~\cite{MaurerPontil2021}, but the hypotheses are different: here the assumptions are expressed through a log-Sobolev constant and an explicit gradient-growth bound, whereas~\cite{MaurerPontil2021} uses Orlicz norms of centered conditional versions. The specialization to leave-one-out cross-validation occurs in \cref{thm:main}, where the general concentration input is combined with the three-term LOO decomposition. The conditioned extension \cref{thm:data-dependent} uses the Bobkov-Nayar-Tetali restricted-measure estimate in \cref{eq:BNT}. The moment estimate \cref{lemma:moments} is used for \cref{cor:loo-moments}.

\begin{theorem}\label{thm:subexp-concentration}
	Let $X_1,\ldots,X_n \iid \mu$ be $\R$-valued random variables
	satisfying, for some $\lambda>0$,
	\begin{align*}
		\E[e^{s(X_i-\E[X_i])}] \leq e^{\frac{\lambda^2 s^2}{2}}, \quad s \leq \frac{1}{\lambda}.
	\end{align*}
	Then for any $\eps > 0$ we get, for $\overline{X}_n = \frac{1}{n} \sum_{i=1}^n X_i$, that
	\begin{align*}
		\P(\overline{X}_n - \E[\overline{X}_n] \geq \eps) \leq e^{-\frac{\eps^2 n}{2 \lambda^2}} \vee e^{- \frac{\eps n}{2\lambda}}.
	\end{align*}
\end{theorem}

The following lemma is the key estimate for our proofs.
\begin{lemma}\label{lemma:quadratic}
Let $F:\R^{(k+m)r}\to\R$ be locally Lipschitz and suppose, for Lebesgue-almost every $x$, that
\begin{equation}\label{eq:gradient-growth}
 \|\nabla F\|(x)\leq a+b\|x\|,\qquad a,b\geq0.
\end{equation}
If $b>0$, then
\begin{equation}\label{eq:gradient-mgf}
 \log\E[e^{\lambda(F-\E[F])}]\leq2\sigma^2(\mu)\lambda^2(a^2+2b^2r\E[\|Z\|^2]),\qquad |\lambda|\leq\frac1{4\sigma^2(\mu)b}.
\end{equation}
If $b=0$, then $F$ is $a$-Lipschitz and
\begin{equation}\label{eq:lipschitz-mgf}
 \log\E[e^{\lambda(F-\E[F])}]\leq\frac{\sigma^2(\mu)a^2\lambda^2}{2},\qquad\lambda\in\R.
\end{equation}
\end{lemma}
\begin{remark}
	In the case of a Lipschitz function we have $b=0$, and the above is valid for all $\lambda$.
\end{remark}
\begin{remark}\label{remark:no-concentration}
Suppose $F$ is a function of $D=(Z_1,\ldots,Z_n)$ with $\|\nabla_D F(D)\|\leq\delta_2(n)\|D\|$, so that $a=0$ and $b=\delta_2(n)$ in \cref{lemma:quadratic}. The Poincaré inequality from \cref{sec:main:ass} gives
\begin{align*}
 \V[F]\leq\sigma^2(\mu)\E[\|\nabla_D F\|^2]
 \leq\sigma^2(\mu)n\delta_2(n)^2\E[\|Z\|^2].
\end{align*}
For a fixed observation law, the condition $\delta_2(n)=o(n^{-1/2})$ makes this upper bound tend to zero. The assumption $\delta_2(n)=O(n^{-1/2})$ alone gives only an $O(1)$ upper bound and therefore does not establish concentration as $n\to\infty$; concentration may still follow from a sharper estimate.
\end{remark}
\begin{proof}[Proof of \cref{lemma:quadratic}]
Since $\sigma^2(\mu)<\infty$, we get for any smooth $f$ from \cref{eq:lsi} that
\begin{align*}
 \Ent(e^f)\leq\frac{\sigma^2(\mu)}{2}\E[\|\nabla f\|^2e^f].
\end{align*}
Now defining $\Gamma(f)=\sigma^2(\mu)\|\nabla f\|^2$, we have by~\cite{B-G} or~\cite[Theorem 2.7]{Ledoux} that
\begin{align*}
 \E[e^{f-\E[f]}]\leq\E[e^{\Gamma(f)}].
\end{align*}
We first bound the exponential moment of $\|X\|^2$ for $X\sim\mu^{\otimes r}$. For $0\leq s<(2\sigma^2(\mu))^{-1}$, we claim that
\begin{equation}\label{eq:norm-mgf}
 \log\E[e^{s\|X\|^2}]\leq\frac{s r\E[\|Z\|^2]}{1-2\sigma^2(\mu)s}.
\end{equation}
Put $Q_N(x)=N(1-e^{-\|x\|^2/N})$ and $\psi_N(s)=\log\E[e^{sQ_N(X)}]$. Then $Q_N\uparrow\|x\|^2$ and $\|\nabla Q_N\|^2\leq4Q_N$. Applying \cref{eq:lsi} to $e^{sQ_N/2}$ under the product law and dividing by $\E[e^{sQ_N(X)}]$ gives
\begin{align*}
 s\psi_N'(s)-\psi_N(s)\leq2\sigma^2(\mu)s^2\psi_N'(s).
\end{align*}
Integrating the resulting inequality for $\psi_N'/\psi_N$ from zero to $s$, using $\psi_N(s)/s\to\E[Q_N]$ at zero, gives $\psi_N(s)\leq s\E[Q_N]/(1-2\sigma^2(\mu)s)$. The case $\E[Q_N]=0$ is immediate. Now $Q_N(X)\uparrow\|X\|^2$ and $\E[Q_N]\uparrow r\E[\|Z\|^2]$, so monotone convergence proves \cref{eq:norm-mgf}.

First suppose that $b>0$ and $F$ is bounded and smooth. Fix $|\lambda|\leq(4\sigma^2(\mu)b)^{-1}$. Applying the Bobkov--G\"otze bound to $f=\lambda F$, using $(a+b\|X\|)^2\leq2a^2+2b^2\|X\|^2$, and then applying \cref{eq:norm-mgf} gives
\begin{align*}
 \log\E[e^{\lambda(F-\E[F])}]
 &\leq2\sigma^2(\mu)\lambda^2
    \left(a^2+\frac{b^2r\E[\|Z\|^2]}{1-4(\sigma^2(\mu))^2\lambda^2b^2}\right)
 \leq2\sigma^2(\mu)\lambda^2(a^2+2b^2r\E[\|Z\|^2]).
\end{align*}
Here $4(\sigma^2(\mu))^2\lambda^2b^2\leq1/4$, and \cref{eq:norm-mgf} also gives the exponential integrability needed in the first step. This proves \cref{eq:gradient-mgf} for both signs of $\lambda$.

The gradient bound implies the segment estimate
\begin{equation}\label{eq:segment}
 |F(x)-F(y)|\leq\|x-y\|\left(a+\frac b2(\|x\|+\|y\|)\right).
\end{equation}
Taking $y=0$ in \cref{eq:segment} gives $|F(X)|\leq|F(0)|+a\|X\|+b\|X\|^2/2$, so $F$ is integrable.

If $b=0$ and $F$ is bounded and smooth, then $F$ is $a$-Lipschitz by \cref{eq:segment}. Applying \cref{eq:lsi} to $e^{\lambda F/2}$ and writing $\psi(\lambda)=\log\E[e^{\lambda F}]$ gives $\lambda\psi'(\lambda)-\psi(\lambda)\leq\sigma^2(\mu)a^2\lambda^2/2$. Divide by $\lambda^2$ and integrate from zero to $\lambda$, using $\psi(\lambda)/\lambda\to\E[F]$; applying the same argument to $-F$ proves \cref{eq:lipschitz-mgf} for both signs of $\lambda$.

For a general $F$, the result follows by truncation and smooth approximation.
\end{proof}

\begin{corollary}\label{cor:quadconc}
Suppose that the assumptions of \cref{lemma:quadratic} hold. If $b>0$, then for every $t>0$,
\begin{align*}
 \P(F-\E[F]>t)
 \leq e^{-\frac{t^2}{8\sigma^2(\mu)(a^2+2b^2r\E[\|Z\|^2])}}
 \vee e^{-\frac{t}{8\sigma^2(\mu)b}}
 =\Theta_r(t;a,b).
\end{align*}
The same estimate holds for $\P(\E[F]-F>t)$. If $b=0$, both one-sided probabilities are bounded by $e^{-t^2/(2\sigma^2(\mu)a^2)}$.
\end{corollary}
\begin{proof}
Suppose first that $b>0$. Put $V=2\sigma^2(\mu)(a^2+2b^2r\E[\|Z\|^2])$ and $\Lambda=(4\sigma^2(\mu)b)^{-1}$. If $V=0$, then $F$ is constant almost surely and the assertion is immediate. By applying \cref{lemma:quadratic} and using Markov's inequality we get for $0<\lambda\leq\Lambda$ that
\begin{align*}
 \P(F-\E[F]>t)\leq e^{-\lambda t+V\lambda^2}.
\end{align*}
The minimizer is given by $\lambda^*=t/(2V)$, which by plugging into the preceding display yields, provided that $\lambda^*\leq\Lambda$,
\begin{align*}
 \P(F-\E[F]>t)\leq e^{-t^2/(4V)}.
\end{align*}
On the other hand, in the case that $\lambda^*>\Lambda$, we choose $\hat\lambda=\Lambda$, from which plugging into the preceding display gives
\begin{align*}
 \P(F-\E[F]>t)\leq e^{-\hat\lambda t+V\hat\lambda^2}\leq e^{-\hat\lambda t/2},
\end{align*}
since $V\hat\lambda\leq t/2$. Since $4V=8\sigma^2(\mu)(a^2+2b^2r\E[\|Z\|^2])$ and $\Lambda/2=(8\sigma^2(\mu)b)^{-1}$, the two cases give precisely the two exponents in $\Theta_r(t;a,b)$. Combining both estimates gives the upper-tail bound for every $t>0$. Apply the same calculation to $-F$ for the opposite tail. When $b=0$ and $a>0$, \cref{eq:lipschitz-mgf} and Chernoff's inequality, optimized at $\lambda=t/(\sigma^2(\mu)a^2)$, give the stated Gaussian bound. If $a=b=0$, \cref{eq:segment} makes $F$ constant and the tail probabilities vanish.
\end{proof}

The following moment estimates follow from the log-Sobolev argument of~\cite{AidaStroock1994}.
\begin{lemma}\label{lemma:moments}
Let $X\sim\mu^{\otimes r}$ and $2\leq p<\infty$. Then
\begin{equation}\label{eq:norm-moment}
 \bigl\|\|X\|\bigr\|_{L^p(\mu^{\otimes r})}\leq Q_{r,p}.
\end{equation}
If $F$ is locally Lipschitz and $\|\nabla F\|(x)\leq a+b\|x\|$ for $a,b\geq0$, then
\begin{equation}\label{eq:gradient-moment}
 \|F-\E[F]\|_{L^p(\mu^{\otimes r})}
 \leq\sqrt{\sigma^2(\mu)(p-1)}\left(a+bQ_{r,p}\right).
\end{equation}
\end{lemma}
\begin{proof}
For a bounded smooth function $g$ and $q\geq2$, differentiate $\|g\|_{L^q(\mu^{\otimes r})}^2=(\E[|g|^q])^{2/q}$. Applying the log-Sobolev inequality to $|g|^{q/2}$, and then H\"older's inequality, gives
\begin{align*}
 \frac{d}{dq}\|g\|_{L^q(\mu^{\otimes r})}^2
 &=\frac{2}{q^2}\bigl(\E[|g|^q]\bigr)^{2/q-1}\Ent(|g|^q)\\
 &\leq \sigma^2(\mu)\bigl(\E[|g|^q]\bigr)^{2/q-1}
 \E[|g|^{q-2}\|\nabla g\|^2]\\
 &\leq \sigma^2(\mu)\|\nabla g\|_{L^q(\mu^{\otimes r})}^2.
\end{align*}
The last step uses
\begin{align*}
 \E[|g|^{q-2}\|\nabla g\|^2]
 \leq(\E[|g|^q])^{(q-2)/q}(\E[\|\nabla g\|^q])^{2/q}.
\end{align*}
If $g=0$ almost surely, the estimate below is immediate; a smooth regularization of the absolute value justifies the calculation at zeros. Since $\|\nabla g\|_{L^q(\mu^{\otimes r})}\leq\|\nabla g\|_{L^p(\mu^{\otimes r})}$ for $q\leq p$, integration from $2$ to $p$ yields
\begin{equation}\label{eq:moment-interpolation}
 \|g\|_{L^p(\mu^{\otimes r})}^2\leq\|g\|_{L^2(\mu^{\otimes r})}^2
 +\sigma^2(\mu)(p-2)\|\nabla g\|_{L^p(\mu^{\otimes r})}^2.
\end{equation}
Apply \cref{eq:moment-interpolation} first to bounded smooth approximations of $x\mapsto\|x\|$. Their gradients are bounded by one and their $L^2(\mu^{\otimes r})$ norms converge to $(r\E[\|Z\|^2])^{1/2}$. Fatou's lemma therefore gives
\begin{align*}
 \bigl\|\|X\|\bigr\|_{L^p(\mu^{\otimes r})}^2
 \leq r\E[\|Z\|^2]+\sigma^2(\mu)(p-2)=Q_{r,p}^2,
\end{align*}
which is \cref{eq:norm-moment}. For bounded smooth $F$, apply \cref{eq:moment-interpolation} to $g=F-\E[F]$. The Poincar\'e inequality gives $\|g\|_{L^2(\mu^{\otimes r})}^2\leq \sigma^2(\mu)\|\nabla F\|_{L^2(\mu^{\otimes r})}^2\leq \sigma^2(\mu)\|\nabla F\|_{L^p(\mu^{\otimes r})}^2$. Hence
\begin{align*}
 \|F-\E[F]\|_{L^p(\mu^{\otimes r})}^2
 \leq \sigma^2(\mu)(p-1)\|\nabla F\|_{L^p(\mu^{\otimes r})}^2.
\end{align*}
By the gradient assumption and Minkowski's inequality, $\|\nabla F\|_{L^p(\mu^{\otimes r})}\leq a+b\bigl\|\|X\|\bigr\|_{L^p(\mu^{\otimes r})}\leq a+bQ_{r,p}$, proving \cref{eq:gradient-moment} for bounded smooth $F$. For general $F$, use the truncations and mollifications from the proof of \cref{lemma:quadratic}. Their gradients are bounded by $a+b\eta+b\|x\|$. First let $\eta\downarrow0$ with the truncation level fixed, and then let that level tend to infinity. The means converge in both steps because $F$ is integrable by \cref{eq:segment}; Fatou's lemma passes the $L^p(\mu^{\otimes r})$ bound to $F-\E[F]$.
\end{proof}

\subsection{Proof of \texorpdfstring{\cref{thm:main}}{Theorem 2}}
Recall that we defined $f_D(z) = L(T_D(x),y)$, for $z = (x,y)$, and $f_i(z) = L(T_{D_{(-i)}}(x),y)$. We begin by decomposing the error into the following components
\begin{align*}
	L_{{LOO}} - \hat L_{{LOO}}
	=&
	\frac{1}{n} \sum_{i=1}^n \left ( \E[f_i(Z_i) \mid Z_i] -f_i(Z_i)\right )
	\\
	&+\frac{1}{n} \sum_{i=1}^n \left(\E[f_i(Z_i)]-\E[f_i(Z_i) \mid Z_i] \right )
	\\
	&+\frac{1}{n} \sum_{i=1}^n \left(\E[f_D(Z) \mid D]-\E[f_i(Z_i)]\right)
	\\
	=&
	I_1 + I_2 + I_3+\E[L_{{LOO}}]-\E[f(Z)].
\end{align*}
Here $f(z)=\E[f_i(z)]$ for every $i$, and $I_3=L_{{LOO}}-\E[L_{{LOO}}]$. By using the union bound and $|\E[L_{{LOO}}]-\E[f(Z)]|\leq\delta_3(n)$,
\begin{align*}
	\P(L_{{LOO}}-\hat L_{{LOO}}>\eps) \leq \P(I_1 > \eps/3)+\P(I_2 > \eps/3)+\P(I_3 > \eps/3-\delta_3(n))
\end{align*}
it suffices to consider the terms $I_1$, $I_2$, and $I_3$ separately. Consider the term $I_1$ first.
We have
\begin{align*}
	\P(I_1 > \eps/3) \leq \sum_{i=1}^n \P(\E[f_i(Z_i) \mid Z_i] -f_i(Z_i) > \eps/3).
\end{align*}

Now for $\mu$-almost every fixed $z$, we may apply \cref{cor:quadconc} together with \cref{assumption:standing} to the negative fitted loss to obtain
\begin{align*}
	\P(\E[f_i(z)]-f_i(z) > \eps/3) \leq \theta_1(n,\eps/3,z),
\end{align*}
where $\theta_1$ is from \cref{eq:theta1-1}.
Hence, by using the towering property for conditional expectation, it follows that
\begin{align*}
	\P(\E[f_i(Z_i) \mid Z_i]-f_i(Z_i) > \eps/3) \leq \E\theta_1(n,\eps/3,Z),
\end{align*}
and hence
\begin{align*}
	\P(I_1 > \eps/3) \leq n \E\theta_1(n,\eps/3,Z).
\end{align*}
Consider next the term $I_2$ and denote $f(z) = \E[f_i(z)]$. We have
\begin{align*}
	\P(I_2 > \eps/3)
	=
	\P\left (\E[f(Z)]-\frac{1}{n} \sum_{i=1}^n f(Z_i) > \eps/3 \right ).
\end{align*}
	The expected-loss bound \cref{eq:expected-loss-gradient} and Cauchy--Schwarz imply
\begin{align*}
 \|\nabla_D I_2\|(D)\leq\frac{c_{n-1}}{\sqrt n}+\frac{d_{n-1}}n\|D\|.
\end{align*}
For local slopes, the same estimate follows by summing the segment bounds \cref{eq:segment} for $f$ and passing to the limit. Hence, by \cref{cor:quadconc},
\begin{align*}
 \P(I_2>\eps/3)\leq\theta_2(n,\eps/3).
\end{align*}

It remains to consider the term $I_3$. Now, by integrating \cref{eq:segment} in the independent test variable,
\begin{align*}
 |\E_Z[f_D(Z)]-\E_Z[f_{D'}(Z)]|
 &\leq\E_Z[|f_D(Z)-f_{D'}(Z)|]\\
 &\leq\|D-D'\|\left[\E[\delta_1(n,Z)]+\frac{\E[\delta_2(n,Z)]}2(\|D\|+\|D'\|)\right].
\end{align*}
Consequently, the conditional risk is locally Lipschitz and
\begin{align*}
 \|\nabla_D L_{{LOO}}\|(D)\leq\E[\delta_1(n,Z)]+\E[\delta_2(n,Z)]\|D\|.
\end{align*}
Hence, again by \cref{cor:quadconc}, we get
\begin{align*}
 \P(I_3>\eps/3-\delta_3(n))\leq\theta_3(n,\eps/3-\delta_3(n)).
\end{align*}
This handles the term $I_3$ as well, and collecting all bounds together proves \cref{eq:main-bound}.

For the uniform bound, $f(Z_i)-f_i(Z_i)$ is centered conditional on $Z_i$. If $b_{n-1}>0$, \cref{lemma:quadratic} gives, for $|\lambda|\leq(4\sigma^2(\mu)b_{n-1})^{-1}$,
\begin{align*}
 \E[e^{\lambda(f(Z_i)-f_i(Z_i))}\mid Z_i]
 \leq\exp\left[2\sigma^2(\mu)\lambda^2
 \bigl(a_{n-1}^2+2b_{n-1}^2(n-1)\E[\|Z\|^2]\bigr)\right].
\end{align*}
Since $I_1=n^{-1}\sum_i(f(Z_i)-f_i(Z_i))$, convexity gives
\begin{equation}\label{eq:average-mgf}
 \E[e^{\lambda I_1}]\leq\frac1n\sum_{i=1}^n\E[e^{\lambda(f(Z_i)-f_i(Z_i))}].
\end{equation}
Thus $I_1$ satisfies the same exponential-moment bound. The Chernoff calculation in the proof of \cref{cor:quadconc} gives $\P(I_1>\eps/3)\leq\Theta_{n-1}(\eps/3;a_{n-1},b_{n-1})$. When $b_{n-1}=0$, \cref{eq:lipschitz-mgf} gives the same conclusion. Together with the bounds for $I_2,I_3$, this proves \cref{eq:uniform-main}. Reversing all three centered terms proves the bounds for the opposite sign.

\subsection{Proof of \texorpdfstring{\cref{thm:simplified}}{Theorem 1}}
\begin{proof}
Use the decomposition \cref{eq:decompose:error}. Put $t=(\eps-\delta_3(n))/3>0$. In the Lipschitz case, the proof of \cref{thm:main} gives, by the $b=0$ part of \cref{cor:quadconc},
\begin{align*}
 \P(I_1>t)&\leq n\E\left[e^{-\frac{t^2}{2\sigma^2(\mu)\delta_1^2(n-1,Z)}}\right],\\
 \P(I_2>t)&\leq e^{-\frac{nt^2}{2\sigma^2(\mu)c_{n-1}^2}},\qquad
 \P(I_3>t)\leq e^{-\frac{t^2}{2\sigma^2(\mu)(\E[\delta_1(n,Z)])^2}}.
\end{align*}
The deterministic term in \cref{eq:decompose:error} has absolute value at most $\delta_3(n)$ by \cref{eq:delta3}. Thus $L_{{LOO}}-\hat L_{{LOO}}>\eps$ implies that at least one of $I_1,I_2,I_3$ exceeds $t$. The union bound gives the first estimate in \cref{thm:simplified}.

Now suppose $\delta_1(n-1,z)\leq a_{n-1}$ for almost every $z$. Conditional on $Z_i=z$, the deleted fit uses $n-1$ independent observations and its centered loss has Lipschitz constant at most $a_{n-1}$. Applying \cref{eq:lipschitz-mgf} conditionally, then averaging over $Z_i$ and using the convexity estimate \cref{eq:average-mgf}, gives
\begin{align*}
 \E[e^{\lambda I_1}]\leq e^{\sigma^2(\mu)a_{n-1}^2\lambda^2/2},\qquad \lambda\in\R.
\end{align*}
The gradient bounds for $I_2$ and $I_3$ in the proof of \cref{thm:main} become the global Lipschitz bounds $c_{n-1}/\sqrt n$ and $\E[\delta_1(n,Z)]$ when $\delta_2=0$. Applying \cref{eq:lipschitz-mgf} to these centered terms gives
\begin{align*}
 \E[e^{\lambda I_j}]\leq e^{\sigma^2(\mu)s_j^2\lambda^2/2},\qquad
 s_1=a_{n-1},\quad s_2=\frac{c_{n-1}}{\sqrt n},\quad s_3=\E[\delta_1(n,Z)].
\end{align*}
Here $s_1+s_2+s_3=s_n$. If $s_j=0$, the centered variable $I_j$ vanishes almost surely and may be omitted. If $s_n>0$, H\"older's inequality with exponents $s_n/s_j$ for the remaining indices yields
\begin{align*}
 \E[e^{\lambda(I_1+I_2+I_3)}]
 \leq\prod_{j:s_j>0}\left(\E[e^{\lambda s_n I_j/s_j}]\right)^{s_j/s_n}
 \leq e^{\sigma^2(\mu)s_n^2\lambda^2/2}.
\end{align*}
Indeed, the reciprocal exponents sum to $\sum_{j:s_j>0}s_j/s_n=1$, and each factor contributes $\sigma^2(\mu)s_js_n\lambda^2/2$ to the logarithm of the product. Chernoff's inequality therefore gives
\begin{align*}
 \P(I_1+I_2+I_3>t)\leq e^{-\frac{t^2}{2\sigma^2(\mu)s_n^2}}.
\end{align*}
Since the remaining deterministic term is bounded by $\delta_3(n)$, this proves \cref{eq:uniform-simple}. The same bound holds for the negative of the sum and gives the opposite one-sided estimate used in the remark. If $s_n=0$, all three centered terms vanish almost surely and the conclusion holds with right-hand side zero.

Finally, under the stated bounds with constant $A$, choose $a_{n-1}=A/\sqrt{n-1}$ and note that $\E[\delta_1(n,Z)]\leq A/\sqrt n$. For $n\geq2$ this gives
\begin{align*}
 s_n\leq\frac{(1+\sqrt2)A+\sup_r c_r}{\sqrt n}=:\frac{C_0}{\sqrt n}.
\end{align*}
Take $n$ large enough that $\delta_3(n)\leq\eps/2$. In the one-sided version of the preceding bound, choose $t=\eps-\delta_3(n)\geq\eps/2$. If $C_0>0$, the exponent is at most $-\eps^2n/(8\sigma^2(\mu)C_0^2)$, which proves \cref{eq:one-sided-simple}; if $C_0=0$, the centered error vanishes almost surely.
\end{proof}

\subsection{Proof of \texorpdfstring{\cref{cor:loo-moments}}{the moment bound}}
\begin{proof}
Use the decomposition $H=I_1+I_2+I_3+\E[L_{{LOO}}]-\E[f(Z)]$ from \cref{eq:decompose:error}. The first three terms are centered, so $H-\E[H]=I_1+I_2+I_3$. We bound their $L^p(\mu^{\otimes n})$ norms separately.

Consider first $I_1$. Conditional on $Z_i=z$, the deleted dataset $D_{(-i)}$ has law $\mu^{\otimes(n-1)}$ and is independent of $z$. The variable $f(z)-f_i(z)$ is centered under this conditional law, and \cref{eq:gradient-condition} bounds the local slope of $f_i(z)$ by $\delta_1(n-1,z)+\delta_2(n-1,z)\|D_{(-i)}\|$. Hence \cref{eq:gradient-moment} gives
\begin{align*}
 \bigl(\E[|f(Z_i)-f_i(Z_i)|^p\mid Z_i=z]\bigr)^{1/p}
 \leq\sqrt{\sigma^2(\mu)(p-1)}
 \left[\delta_1(n-1,z)+Q_{n-1,p}\delta_2(n-1,z)\right].
\end{align*}
Integrate the $p$th power with respect to $z$. Minkowski's inequality applied to the two terms in brackets gives
\begin{align*}
 \|f(Z_i)-f_i(Z_i)\|_{L^p(\mu^{\otimes n})}
 \leq\sqrt{\sigma^2(\mu)(p-1)}\bigl[\|\delta_1(n-1,Z)\|_{L^p(\mu)}
 +Q_{n-1,p}\|\delta_2(n-1,Z)\|_{L^p(\mu)}\bigr].
\end{align*}
Applying Minkowski's inequality once more to their average,
\begin{align*}
 \|I_1\|_{L^p(\mu^{\otimes n})}&\leq\frac1n\sum_i\|f(Z_i)-f_i(Z_i)\|_{L^p(\mu^{\otimes n})}\\
 &\leq\sqrt{\sigma^2(\mu)(p-1)}\bigl[\|\delta_1(n-1,Z)\|_{L^p(\mu)}
 +Q_{n-1,p}\|\delta_2(n-1,Z)\|_{L^p(\mu)}\bigr].
\end{align*}
For $I_2$, the proof of \cref{thm:main} established the local-slope bound $\|\nabla_D I_2\|(D)\leq c_{n-1}/\sqrt n+(d_{n-1}/n)\|D\|$. Apply \cref{eq:gradient-moment} on the product space carrying $D$ to get
\begin{align*}
 \|I_2\|_{L^p(\mu^{\otimes n})}\leq\sqrt{\sigma^2(\mu)(p-1)}
 \left[\frac{c_{n-1}}{\sqrt n}+\frac{d_{n-1}}nQ_{n,p}\right].
\end{align*}
For $I_3$, the conditional-risk estimate from the same proof gives $\|\nabla_D I_3\|(D)\leq\E[\delta_1(n,Z)]+\E[\delta_2(n,Z)]\|D\|$. A second application of \cref{eq:gradient-moment} yields
\begin{align*}
 \|I_3\|_{L^p(\mu^{\otimes n})}\leq\sqrt{\sigma^2(\mu)(p-1)}
 \left[\E[\delta_1(n,Z)]+\E[\delta_2(n,Z)]Q_{n,p}\right].
\end{align*}
Finally, Minkowski's inequality gives
\begin{align*}
 \|H-\E[H]\|_{L^p(\mu^{\otimes n})}
 \leq\|I_1\|_{L^p(\mu^{\otimes n})}+\|I_2\|_{L^p(\mu^{\otimes n})}+\|I_3\|_{L^p(\mu^{\otimes n})}.
\end{align*}
Substituting the three estimates gives exactly $B_{n,p}$ in \cref{eq:loo-moment-scale}, proving \cref{eq:loo-moment}.
\end{proof}

\subsection{Proof of \texorpdfstring{\cref{thm:data-dependent}}{Theorem 3}}
\begin{proof}
We first estimate the fluctuations of $H$ after conditioning on $D\in K$. The log-Sobolev inequality tensorizes, so the product law $\mu^{\otimes n}$ has the same admissible constant $\sigma^2(\mu)$ as $\mu$. By \cref{lemma:sgls}, its sub-Gaussian constant is at most $\sigma^2(\mu)$. Applying \cref{eq:BNT} to the event $D\in K$ gives
\begin{align*}
 \sigma_{SG}^2(\P_K)
 \leq c_{\mathrm{BNT}}\log\frac{e}{\P(D\in K)}\,\sigma^2(\mu)=C_K.
\end{align*}
The assumption \cref{eq:data-dependent-condition} says that $H$, viewed as a function on the metric space $K$, is $L_K$-Lipschitz. Thus, for every $\lambda\in\R$,
\begin{align*}
 \E_K\left[e^{\lambda(H-\E_K[H])}\right]
 \leq e^{C_KL_K^2\lambda^2/2}.
\end{align*}
If $L_K>0$, Chernoff's inequality, optimized at $\lambda=t/(C_KL_K^2)$, yields
\begin{align*}
 \P_K(H-\E_K[H]>t)\leq e^{-\frac{t^2}{2C_KL_K^2}},\qquad t>0.
\end{align*}
The same estimate holds for the opposite sign. When $L_K=0$, $H$ is constant on $K$ and both conditional tail probabilities vanish for $t>0$.

We next compare the conditional mean with the original mean. Since $\P(D\in K)=1-p$ and $\E[\mathbf1_K-(1-p)]=0$,
\begin{align*}
 \E_K[H]-\E[H]
 &=\frac{\E[(H-\E[H])(\mathbf1_K-(1-p))]}{1-p},\\
 \E[(\mathbf1_K-(1-p))^2]&=p(1-p).
\end{align*}
By Cauchy--Schwarz and the assumption $\|H\|_{L^2(\mu^{\otimes n})}\leq B$,
\begin{align*}
 |\E_K[H]-\E[H]|
 \leq\|H-\E[H]\|_{L^2(\mu^{\otimes n})}\sqrt{\frac{p}{1-p}}
 \leq B\sqrt{\frac{p}{1-p}}=\eta_K.
\end{align*}
Together with $|\E[H]|\leq\delta_3(n)$, this gives $|\E_K[H]|\leq\delta_3(n)+\eta_K$.

Finally, we return to the original law. For $\eps>\delta_3(n)+\eta_K$, conditioning on whether $D$ belongs to $K$ gives
\begin{align*}
 \P(H>\eps)
 &\leq\P(D\notin K)+\P(D\in K)\,
   \P_K\bigl(H-\E_K[H]>\eps-\delta_3(n)-\eta_K\bigr)\\
 &\leq\P(D\notin K)
   +\exp\left[-\frac{(\eps-\delta_3(n)-\eta_K)^2}{2C_KL_K^2}\right],
\end{align*}
which is \cref{eq:conditioned-bound}. Applying the same argument to $-H$ proves the opposite one-sided estimate. For the absolute error, use the two-sided conditional bound before adding $\P(D\notin K)$, so this probability is counted once.
\end{proof}

\section{Conclusions}\label{sec:conclusion}
In this article we have provided concentration inequalities for {LOO} cross-validation under a general framework. Our approach is applicable for data arising from a distribution that satisfies the logarithmic Sobolev inequality, providing us a relatively rich class of distributions. The main message is that a weak delete-one error-stability condition, together with a gradient-based control of the fitted loss as a function of the data and of its expectation as a function of the test point, is sufficient to obtain concentration of the {LOO} estimate around the corresponding conditional risk.

There are also limitations to the framework. First, the logarithmic Sobolev assumption is a tail condition on the data distribution. It is well suited to Gaussian-type unbounded data and tensorizes cleanly for i.i.d.~samples, but it does not cover all distributions of interest. In the case of bounded random variables, one can apply well-known concentration inequalities for bounded random variables when the corresponding coordinatewise sensitivity bounds hold. Second, the gradient-based stability condition must be verified for the particular estimator and loss under consideration. In examples such as ordinary linear regression, the required control may fail globally and has to be recovered by conditioning on a suitable high-probability set. The usefulness of the conditioned result therefore depends on finding a set $K$ on which the estimator is stable and on controlling the cost of the restriction.

The bounds in \cref{thm:simplified,thm:main,thm:data-dependent} are primarily large-deviation bounds. They show that large errors are exponentially unlikely once the stability parameters are small enough, but they are not intended to be sharp performance guarantees for every specific estimator. For particular estimators and particular distributions, one may be able to obtain sharper bounds by exploiting additional structure such as convexity, explicit formulas, permutation symmetry, or exact leave-one-out identities. Our results should therefore be viewed as a general stability-based mechanism, rather than as a replacement for estimator-specific analyses.

We have also illustrated the applicability of our method by considering several examples. The same methodology could be used in other practical estimation schemes as well. While we have restricted our study to the case of {LOO} cross-validation, similar ideas should apply to other cross-validation procedures. For leave-$k$-out cross-validation, the main changes would be to replace the single deleted point by a deleted block of size $k$, to track the resulting blockwise bias term in place of $\delta_3$, and to adjust the union and averaging steps in the proof according to the number and dependence structure of the folds. For $V$-fold cross-validation one would similarly need a stability condition for replacing one fold by an independent test block.

Our approach could also be used for model selection under the same stability assumptions. The quantity of interest in this case is not the two cross-validated losses $\hat L^{1}_{LOO}$ and $\hat L^{2}_{LOO}$ separately, but rather their difference $\hat L^{1}_{LOO} - \hat L^{2}_{LOO}$, or its sign if one is only interested in which model performs better. Often one expects this difference to concentrate better than either individual term, see for instance~\cite{Arlot0,Arlot}. The bounds in \cref{thm:simplified,thm:main,thm:data-dependent} can be applied to the difference when their gradient and delete-one bias hypotheses hold for it. If the difference has better stability in the sense of \cref{definition:standing}, and the other hypotheses remain valid, then it satisfies a stronger concentration bound around the corresponding conditional-risk difference. This in turn justifies basing model selection directly on the sign or size of $\hat L^{1}_{LOO} - \hat L^{2}_{LOO}$ when the absolute conditional-risk difference exceeds the deviation level in the bound.

\section*{Acknowledgments}
\addcontentsline{toc}{section}{Acknowledgments}
{
B.A. was supported by the Swedish Research Council dnr: 2019--04098. We thank the anonymous reviewers for their careful reading and detailed comments, which helped us correct errors in earlier versions and improve the statements, proofs, and presentation.
}

\bibliography{main}
\end{document}